\documentclass[a4paper,12pt,reqno]{amsart}
\usepackage{amssymb}
\usepackage{latexsym}
\usepackage{amsfonts}
\usepackage{amsmath}
\usepackage{amsthm}
\usepackage{enumitem}
\usepackage{multicol}
\theoremstyle{definition}
\newtheorem{definition}{Definition}[section]
\newtheorem{proposition}[definition]{Proposition}
\newtheorem{theorem}[definition]{Theorem}

\newtheorem{lemma}[definition]{Lemma}

\newtheorem{remark}[definition]{Remark}

\def\<{\mathop{<}}
\def\>{\mathop{>}}

\newcommand{\spt}{\mathrm{spt}\,}

\newcommand{\dist}{\mathrm{dist}\,}
\numberwithin{equation}{section}

\title[Convergence of the Allen-Cahn equation to MCF]{Convergence of the Allen-Cahn equation with constraint to Brakke's mean curvature flow}
\author[K. Takasao]{Keisuke Takasao \\ Graduate School of Mathematical Sciences
University of Tokyo \\ Komaba 3-8-1, Meguro
JP-153-8914 Tokyo
Japan}

\email{takasao@ms.u-tokyo.ac.jp}
\thanks{The author is grateful to Prof. Yoshihiro Tonegawa and Prof. 
Noriaki Yamazaki and Prof. Tomoyuki Suzuki for numerous comments}
\date{}
\begin{document}
\maketitle
\begin{abstract}
In this paper we consider the Allen-Cahn equation with constraint. In 1994, Chen and Elliott~\cite{chenelliott} studied the asymptotic behavior of the solution of the Allen-Cahn equation with constraint. They proved that  the zero level set of the solution converges to the classical solution of the mean curvature flow under the suitable conditions on initial data. In 1993, Ilmanen~\cite{Ilmanen} proved the existence of the mean curvature flow via the Allen-Cahn equation without constraint in the sense of Brakke. We proved the same conclusion for the Allen-Cahn equation with constraint.
\end{abstract}
\section{Introduction}
Let $T>0$ and $\varepsilon \in (0,1)$. In this paper, we consider the following Allen-Cahn equation with constraint:
\begin{equation}
\left\{
\begin{array}{ll}
\partial _t \varphi ^{\varepsilon} -\Delta \varphi ^{\varepsilon} +\dfrac{\partial I_{[-1,1]} (\varphi ^{\varepsilon}) -\varphi ^\varepsilon}{\varepsilon ^2} \ni 0, & (x,t)\in \mathbb{R}^n \times (0,T),  \\
\varphi ^{\varepsilon} (x,0) = \varphi _0 ^{\varepsilon} (x) , & x\in \mathbb{R}^n.
\end{array} 
\right.
\label{ac2}
\end{equation}
Here, $I_{[-1,1]} $ is the indicator function of $[-1,1]$ defined by
\begin{equation*}
I_{[-1,1]} (s)=\left\{
\begin{array}{ll}
0, & \text{if} \ s\in [-1,1] , \\
+\infty , & \text{otherwise},
\end{array} 
\right.
\end{equation*}
and $\partial I_{[-1,1]} $ is the subdifferential of $ I_{[-1,1]} $, that is
\[
  \partial I_{[-1,1]} (s) = \begin{cases}
    \emptyset, & \text{if} \ s<-1 \ \text{or} \ s>1, \\
    [0,\infty), & \text{if} \  s=1, \\
\{ 0\}, & \text{if} \ -1<s<1, \\
(-\infty,0 ], &\text{if} \ s=-1. 
  \end{cases}
\]
Set
\[ \quad \mathcal{G} :=\{ v \in L^\infty (\mathbb{R}^n) : \| v \|_{L^\infty (\mathbb{R}^n)} \leq 1 \} \quad \text{and} \quad \mathcal{K} :=\mathcal{G}\cap H^1 (\mathbb{R}^n). \]
For $\varphi^\varepsilon _0 \in\mathcal{G}$, $\varphi ^\varepsilon \in C(0,T;L^2 (\mathbb{R}^n))$ is called a solution for \eqref{ac2} if the following hold:
\[
\begin{cases}
\varphi ^\varepsilon \in L^2(0,T;H^1 (\mathbb{R}^n)) , \ \partial _t \varphi ^\varepsilon \in  L^2(0,T; ( H^1 (\mathbb{R}^n) )' ), \\
\varphi ^\varepsilon (\cdot, t) \in \mathcal{K} \ \text{a.e.} \ t\in (0,T), \ \varphi^\varepsilon (\cdot ,0) =\varphi^\varepsilon _0 (\cdot),\\
\int_0 ^T \{ \langle \partial _t \varphi ^\varepsilon, v-\varphi ^\varepsilon \rangle +(\nabla \varphi ^\varepsilon ,\nabla (v-\varphi ^\varepsilon)) -\frac{1}{\varepsilon ^2} (\varphi ^\varepsilon,v-\varphi ^\varepsilon) \} \, dt \geq 0\\
\text{for any} \ v\in L^2( 0,T;H^1 (\mathbb{R}^n  ) ), \quad \text{with} \quad v(\cdot ,t) \in \mathcal{K} \ \text{for any} \ t \in (0,T).
\end{cases}
\]
Here $\langle \ ,\, \rangle $ denotes the pairing of $( H^1 (\mathbb{R}^n) )'$ and $H^1 (\mathbb{R}^n)$, and $(\ , \, )$ denotes the inner product in $L^2 (\mathbb{R}^n)$.

Let $\delta \in (0,\frac{1}{2})$. To study \eqref{ac2}, we consider the following equation:
\begin{equation}
\left\{ 
\begin{array}{ll}
\partial _t \varphi ^{\varepsilon,\delta} -\Delta \varphi ^{\varepsilon,\delta} +\dfrac{F'_\delta (\varphi ^{\varepsilon,\delta})}{\varepsilon ^2}=0,& (x,t)\in \mathbb{R}^n \times (0,\infty),  \\
\varphi ^{\varepsilon,\delta} (x,0) = \varphi _0 ^{\varepsilon,\delta} (x) ,  &x\in \mathbb{R}^n,
\end{array} \right.
\label{ac}
\end{equation}
where
\begin{equation*}
F_\delta (s)=\left\{ 
\begin{array}{ll}
\dfrac{1-\delta}{2\delta}\Big( s+\dfrac{1}{1-\delta} \Big)^2, & \text{if}  \ s<-1,\\[0.3cm]
-\dfrac{1}{2}s^2 +\dfrac{1}{2(1-\delta)}, & \text{if} \  |s|\leq 1,\\[0.3cm]
\dfrac{1-\delta}{2\delta}\Big( s-\dfrac{1}{1-\delta} \Big)^2, & \text{if} \ s>1.
\end{array} 
\right.
\end{equation*}
The function $F'_\delta (s)$ is the Yosida approximation of $\partial I _{[-1,1]} (s) -s$. We remark that $F\in C^{1}(\mathbb{R})$, $F_\delta (s) \geq 0$ and $F_\delta (s)=0$ if and only if $s=\pm(1-\delta)^{-1}$. By an argument similar to that in \cite{chenelliott}, the classical solution of \eqref{ac} converges to the solution of \eqref{ac2} under the suitable conditions on initial data as $\delta \to 0$ for any $T>0$.

The purpose of this paper is to prove that the solutions of \eqref{ac2} and \eqref{ac} converge to a weak solution for the mean curvature flow.
Here, a family of hypersurfaces $\{ \Gamma (t) \}_{t\in [0,T)}$ is called the mean curvature flow if the velocity of $\Gamma (t)$ is
\begin{equation}
V_\Gamma = H \quad \text{on} \ \Gamma (t) ,\quad t \in (0,T),
\label{mcf}
\end{equation}
where $H$ is the mean curvature vector of $\Gamma (t)$.
Chen and Elliott~\cite{chenelliott} proved that for a classical solution $\{ \Gamma (t) \}_{t\in [0,T)}$ of the mean curvature flow, there exists a family of functions $\{ \varphi _0 ^\varepsilon \}_{\varepsilon >0}$ such that the zero level set of the solution $\varphi ^{\varepsilon}$ for \eqref{ac2} converges to $\{ \Gamma (t) \}_{t\in [0,T)}$ as $\varepsilon \to 0$. But there is no result for the construction of the global weak solution for the mean curvature flow via \eqref{ac2} or \eqref{ac}. 

In this paper, we consider a weak solution for the mean curvature flow called Brakke's mean curvature flow which we define later~\cite{brakke}. There is a large amount of research on the mean curvature flow~\cite{almgrentaylorwang,giga1991,evans-spruck1,evans-spruck2,evans-spruck3,evans-spruck4} and the connection between the Allen-Cahn equation and the mean curvature flow~\cite{ac,bronsard,Chen1992,evans-soner-souganidis}, so we may mention only a part of them related to Brakke's mean curvature flow and  \eqref{ac2}.
Brakke~\cite{brakke} proved the existence of a Brakke's mean curvature flow by using
geometric measure theory. Ilmanen~\cite{Ilmanen} proved that the singular limit of the Allen-Cahn equation without constraint is a Brakke's mean curvature flow under mild conditions on initial data. The main results of this paper is the same conclusion for \eqref{ac2} and \eqref{ac}. Liu, Sato and Tonegawa~\cite{tonegawa2010}, and Takasao and Tonegawa~\cite{takasaotonegawa} proved that there exists Brakke's mean curvature flow with transport term via the phase field method. Moreover, the regularity of Brakke's mean curvature flow was proved by Kasai and Tonegawa~\cite{kasaitonegawa} and Tonegawa~\cite{tonegawa2014} by
improving on Brakke's partial regularity theorem for mean curvature flow. Recently, Farshbaf-Shaker, Fukao and Yamazaki~\cite{yamazaki2013} characterized the Lagrange multiplier $\lambda ^\varepsilon$ of \eqref{ac2}, where $\lambda ^\varepsilon=\lambda ^\varepsilon (\varphi ^\varepsilon)$ satisfies
\[ (\partial _t \varphi ^\varepsilon ,\psi) +(\nabla \varphi ^\varepsilon , \nabla \psi ) +\frac{1}{\varepsilon ^2} (\lambda ^\varepsilon ,\psi) = \frac{1}{\varepsilon ^2} (\varphi ^\varepsilon ,\psi)  \]
for any $\psi \in H^1 (\mathbb{R}^n)$ and a.e. $t\in (0,T)$. Suzuki, Takasao and Yamazaki~\cite{suzukitakasaoyamazaki} studied the criteria for the standard
forward Euler method to give stable numerical experiments of \eqref{ac2}.

The organization of the paper is as follows. In Section 2 of this paper we set out the basic definitions and explain the main results. In Section 3 we study the monotonicity formula and prove some propositions. In Section 4 we show the existence of limit measure $\mu _t$ which corresponds to $ \Gamma (t) $. In Section 5 we prove the density lower bound of $\mu _t$ and the vanishing of the discrepancy measure $\xi$. In Section 6 we show the main results.
\section{Preliminaries and main results}
We recall some notations from geometric measure theory and refer to \cite{allard,brakke,evansgariepy,federer,simon} for more details. On $\mathbb{R}^n$ we denote the Lebesgue measure by $\mathcal{L}^n$. Define $\omega _n :=\mathcal{L}^n (B_1 (0))$. For $r>0$ and $a\in \mathbb{R}^n$ we define $B_r (a):=\{ x\in \mathbb{R}^n \, | \, |x-a|<r \}$. We denote the space of bounded variation functions on $\mathbb{R}^n$ as $BV (\mathbb{R}^n)$. We write the characteristic function of a set $A \subset \mathbb{R}^n$ as $\chi _{A}$. For a set $A\subset \mathbb{R}^n$ with finite perimeter, we denote the total variation measure of the distributional derivative $\nabla \chi _{A}$ by $\| \nabla \chi _{A} \|$. For $a=(a_1 ,a_2,\dots ,a_n)$, $b=(b_1 ,b_2,\dots ,b_n) \in \mathbb{R}^n$ we denote $a\otimes b :=(a_i b_j)$. For $A=(a_{ij}),B=(b_{ij}) \in \mathbb{R}^{n\times n}$, we define
\begin{equation*}
A:B:=\sum_{i,j=1}^n a_{ij}b_{ij}.
\end{equation*}
Let $G_k (\mathbb{R}^n)$ be the Grassman manifold of unoriented $k$-dimensional subspaces in $\mathbb{R}^n$. Let $S\in G_k (\mathbb{R}^n)$. We also use $S$ to denote the $n$ by $n$ matrix representing the orthogonal projection $\mathbb{R}^n\to S$. Especially, if $k=n-1$ then the projection for $S\in G_{n-1}(\mathbb{R}^n)$ is given by $S=I-\nu\otimes \nu$, where $I$ is the identity matrix and $\nu$ is the unit normal vector of $S$. Let $S^\perp \in G_{n-k}(\mathbb{R}^n)$ be the orthogonal complement of $S$.

We call a Radon measure on $\mathbb{R}^n\times G_k (\mathbb{R}^n)$ a general $k$-varifold in $\mathbb{R}^n$. We denote the set of all general $k$-varifolds by $\mathbf{V}_k(\mathbb{R}^n)$. Let $V \in \mathbf{V}_k (\mathbb{R}^n)$. We define a mass measure of $V$ by
\[ \| V \| (A) := V( (\mathbb{R} ^n \cap A ) \times G_k (\mathbb{R}^n) )  \]
for any Borel set $A \subset \mathbb{R}^n$.
We also denote
\[ \| V \| (\phi ) :=  \int _{ \mathbb{R} ^n  \times G_k (\mathbb{R}^n) } \phi (x) \, dV(x,S) \quad \text{for} \quad \phi \in C_c (\mathbb{R}^n). \]
The first variation $\delta V:C_c ^1 (\mathbb{R}^n ;\mathbb{R}^n)\to \mathbb{R}$ of $V\in \mathbf{V}_k(\mathbb{R}^n)$ is defined by
\[ \delta V(g):= \int _{\mathbb{R} ^n \times G_k (\mathbb{R} ^n)} \nabla g(x) : S \, dV(x,S) \quad \text{for} \quad g \in C_c ^1 (\mathbb{R}^n ;\mathbb{R}^n).\]
We define a total variation $\| \delta V \|$ to be the largest Borel regular measure on $\mathbb{R}^n$ determined by
\[ \|\delta V \| (G) := \sup \{ \delta V(g)  \, | \, g\in C_c ^1 (G;\mathbb{R}^n), \ |g|\leq 1 \} \]
for any open set $G\subset \mathbb{R}^n$. If $\|\delta V \|$ is locally bounded and absolutely continuous with respect to $\|V\|$, by the Radon-Nikodym theorem, there exists a $\|V\|$-measurable function 
$ H(x)$ with values in $\mathbb{R}^n$ such that
\[  \delta V  (g) = - \int _{\mathbb{R} ^n }H(x)  \cdot g(x)\, d\|V\|(x)  \quad \text{for} \quad g\in C_c (\mathbb{R}^n;\mathbb{R}^n). \]
We call $H$ the generalized mean curvature vector of $V$. 

Let $\mathcal{H}^k$ be the $k$-dimensional Hausdorff measure. We call a Radon measure $\mu $ $k$-rectifiable if $\mu $ is represented by $\mu = \theta \mathcal{H} ^k \lfloor M$, that is, $\mu (\phi)=\int _{\mathbb{R}} \phi \, d\mu = \int _M \phi \theta \, d\mathcal{H}^k$ for any $\phi \in C_c (\mathbb{R}^n)$. Here $M$ is countably $k$-rectifiable and $\mathcal{H}^k$-measurable, and $\theta \in L^1 _{loc} (\mathcal{H}^k \lfloor M)$ is positive valued $\mathcal{H}^k$-a.e. on $M$. For a $k$-rectifiable Radon measure $\mu= \theta \mathcal{H} ^k \lfloor M$ we define a unique $k$-varifold $V$ by
\begin{equation}
\int _{\mathbb{R} ^n \times G_k (\mathbb{R} ^n)} \phi (x,S) \, dV(x,S)
:= \int _{\mathbb{R}^n} \phi (x,T_x \mu ) \, d\mu (x) \qquad \text{for} \ \phi \in C_c(\mathbb{R}^n\times G_k (\mathbb{R}^n)), 
\label{natural}
\end{equation}
where $T_x\mu$ is the approximate tangent space of $M$ at $x$. Note that $T_x\mu$ exists $\mathcal{H}^k$-a.e. on $M$ in this assumption, and $\mu =\|V\|$ under this correspondence.
\begin{definition}
Let $\mu$ be a Radon measure on $\mathbb{R}^n$ and $\phi \in C_c ^2 (\mathbb{R}^n ;\mathbb{R}^+)$. We define 
\[ \mathcal{B}(\mu,\phi) :=\int_{\mathbb{R}^n} -\phi |H|^2 + \nabla \phi \cdot (T_x \mu)^{\perp} \cdot H \, d\mu \]
if  $\mu \lfloor \{ \phi >0 \} $ is rectifiable, $ \| \delta V \| \lfloor \{ \phi >0 \}  \ll \mu  \lfloor \{ \phi >0 \} $ and $\int_{\mathbb{R}^n} |H|^2 \,d\mu <\infty$. Here $V$ is a $k$-varifold defined by \eqref{natural} and $H$ is the the generalized mean curvature vector of $V$. If
any one of the condition is not satisfied, then we define $\mathcal{B}(\mu,\phi) :=-\infty$.
\end{definition}
\begin{definition}
A family $\{ \mu _t \}_{t\geq 0}$ of Radon measures is called Brakke's mean curvature flow if
\begin{equation}
\overline{D} _t \mu _t (\phi) \leq \mathcal{B} (\mu _t ,\phi)
\label{brakkeineq}
\end{equation}
is hold for any $\phi \in C_c ^2 (\mathbb{R}^n;\mathbb{R}^+)$ and any $t\geq 0$. Here $ \overline{D} f(t)= \varlimsup_{h\to 0}\frac{f(t+h)-f(t)}{h}$ is the upper derivative.
\end{definition}
\begin{definition}
Let $\varphi ^{\varepsilon ,\delta}$ be a solution for \eqref{ac}. We define a Radon measure $\mu _t ^{\varepsilon,\delta}$ by
\begin{equation}
\mu _t ^{\varepsilon,\delta}(\phi) := \int _{\mathbb{R}^n} \phi \Big( \frac{\varepsilon |\nabla \varphi ^{\varepsilon,\delta}|^2}{2} + \frac{F_\delta (\varphi ^{\varepsilon,\delta} )}{\varepsilon} \Big) dx,
\end{equation}
for any $\phi \in C_c (\mathbb{R}^n)$. 
\end{definition}

For $r \in \mathbb{R}$ we define
\begin{equation}
q^{\varepsilon}(r)
:=\begin{cases} 
-1 , & \text{if} \ r< - \frac{\varepsilon\pi}{2}, \\
\sin \frac{r}{\varepsilon}  ,& \text{if} \ |r| \leq\frac{\varepsilon\pi}{2} , \\
1, & \text{if} \ r>\frac{\varepsilon\pi}{2}
\end{cases}
\end{equation}
and
\begin{equation}
q^{\varepsilon,\delta}(r)
:=\begin{cases} 
\frac{\delta}{1-\delta} e^{\sqrt{ \frac{1-\delta}{\delta} } \sin ^{-1}\sqrt{1-\delta}}    e^{\frac{r}{\varepsilon } \sqrt{ \frac{1-\delta}{\delta}} } -\frac{1}{1-\delta} , & \text{if} \ r< - \varepsilon \sin ^{-1} \sqrt{1-\delta}, \\
\frac{1}{\sqrt{1-\delta}} \sin \frac{r}{\varepsilon}  ,& \text{if} \ |r| \leq\varepsilon \sin ^{-1} \sqrt{1-\delta}, \\
- \frac{\delta}{1-\delta} e^{\sqrt{ \frac{1-\delta}{\delta} } \sin ^{-1}\sqrt{1-\delta}}    e^{- \frac{r}{\varepsilon } \sqrt{ \frac{1-\delta}{\delta}} } +\frac{1}{1-\delta} , & \text{if} \ r>\varepsilon \sin ^{-1} \sqrt{1-\delta}.
\end{cases}
\end{equation}
\begin{remark}
\begin{enumerate}
\item $q^{\varepsilon}\in C^{1,\alpha} (\mathbb{R})$, $q^{\varepsilon,\delta}\in C^{2}(\mathbb{R})$ and for any $\varepsilon >0$ we have
\begin{equation}
\lim _{\delta\to 0} \| q^{\varepsilon,\delta}_0 -q ^{ \varepsilon  } _0\|_{C^{1,\alpha} (\mathbb{R})} =0. 
\label{limq}
\end{equation}
\item $q^{\varepsilon,\delta}$ is a solution for
\begin{equation}
\frac{\varepsilon ( q^{\varepsilon,\delta} _r)^2 }{2} = \frac{F_\delta (q ^{\varepsilon,\delta})}{\varepsilon} \qquad \text{and} \qquad 
q^{\varepsilon,\delta} _{rr}= \frac{F' _\delta (q ^{\varepsilon,\delta})}{\varepsilon ^2}
\label{q}
\end{equation}
with $q^{\varepsilon,\delta}(0)=0, \ q^{\varepsilon,\delta}(\pm \infty)=\pm (1-\delta)^{-1}$, $q^{\varepsilon,\delta} (\pm \varepsilon \sin ^{-1} \sqrt{1-\delta} )=\pm1$ and $ q^{\varepsilon,\delta}_r (r)  >0$ for any $r\in \mathbb{R}$. Moreover we have
\begin{equation}
\sup_{r\in \mathbb{R}, \delta \in(0,\frac{1}{2})} | q^{\varepsilon,\delta}_{r} (r)|  \leq 2\varepsilon ^{-1} \quad \text{and} \quad \sup_{r\in \mathbb{R}, \delta \in(0,\frac{1}{2})} | q^{\varepsilon,\delta}_{rr}(r)|  \leq 2\varepsilon ^{-2}.
\label{q3}
\end{equation}
\item By \eqref{q} we have
\begin{equation}
\begin{split}
&\int_{\mathbb{R}}  \frac{\varepsilon ( q ^{\varepsilon,\delta}_r )^2}{2} + \frac{F_\delta (q^{\varepsilon,\delta} )}{\varepsilon} \,dr = 
\int_{\mathbb{R}} \sqrt{2F_\delta (q^{\varepsilon,\delta} )} q ^{\varepsilon,\delta}_r  \,dr \\
=& \int_{-(1-\delta)^{-1}} ^{(1-\delta)^{-1}} \sqrt{2F_\delta ( s )} \, ds=:\sigma_{\delta}.
\end{split}
\end{equation}
\end{enumerate}
\end{remark}

Let $\Omega^+ _0 \subset \mathbb{R}^n$ be a bounded open set and we denote $\Gamma _0 :=\partial \Omega^+ _0$. 
Throughout this paper, we assume the following:
\begin{enumerate}
\item There exists $D_0>0$ such that
\begin{equation}
\sup_{x\in \mathbb{R}^n, R>0}\frac{\mathcal{H}^{n-1} (\Gamma _0  \cap B_R (x)) }{\omega _{n-1}R^{n-1}} \leq D_0 \quad \text{(Density upper bounds)}.
\label{initialdata1}
\end{equation}
\item There exists a family of open sets $\{ \Omega _0 ^i \}_{i=1} ^\infty$ such that $\Omega _0 ^i$ have a  $C^3$ boundary $\Gamma _0 ^i$ such that $(\Omega ^+ _0 ,\Gamma _0)$ be approximated strongly by $\{ (\Omega _0 ^i ,\Gamma _0 ^i) \} _{i=1} ^\infty$, that is
\begin{equation}
\lim _{i\to \infty} \mathcal{L}^n (\Omega^+ _0 \triangle \Omega_0 ^i) =0 \quad \text{and} \quad
\lim _{i \to \infty} \| \nabla \chi _{\Omega ^i _0} \| = \|\nabla \chi _{\Omega ^+ _0} \| \ \ \text{as measures.}
\label{initialdata2}
\end{equation}
\end{enumerate}
\begin{remark}
If $\Gamma _0$ is $C^1$, then \eqref{initialdata1} and \eqref{initialdata2} are satisfied.
\end{remark}
Let $\{ \varepsilon _i \}_{i=1}^\infty$ and $\{ \delta _i\}_{i=1}^\infty$ be sequences with $\varepsilon _i,\delta _i \downarrow 0$ as $i\to \infty$. For $\Omega _0^i$ we define
\[
   r_{\varepsilon_i}(x)= \begin{cases}
    \dist(x,\Gamma_0 ^i), & x\in \Omega _0^i \\
    -\dist(x,\Gamma_0 ^i), & x\notin \Omega _0 ^i.
  \end{cases}
\]
We remark that $|\nabla r_{\varepsilon_i} |\leq 1$ a.e. $x\in \mathbb{R}^n$ and $r_{\varepsilon_i}$ is smooth near $\Gamma _0 ^i$. Let $\overline{r_{\varepsilon_i}}$ be a smoothing of $r_{\varepsilon_i} $ with $|\nabla\overline{ r_{\varepsilon_i}} |\leq 1$, $|\nabla ^2 \overline{ r_{\varepsilon_i}} |\leq \varepsilon _i ^{-1}$ in $\mathbb{R}^n$ and $\overline{r_{\varepsilon_i}} =r_{\varepsilon_i}$ near $\Gamma _0 ^i$.

Define
\begin{equation}
\varphi ^{\varepsilon_i }_0=q^{\varepsilon_i } (\overline{r_{\varepsilon_i}} (x)) \quad \text{and} \quad \varphi ^{\varepsilon_i ,\delta_j}_0=q^{\varepsilon_i , \delta _j} (\overline{r_{\varepsilon_i}} (x)), \quad i,j\geq 1.
\label{initial}
\end{equation}

%
Let $U \subset \mathbb{R}^n$ be a bounded open set and $Q_T :=U \times (0,T)$ for $T>0$. 
By \eqref{limq}, \eqref{q3} and \eqref{initial} there exists $c_1(i) >0$ such that 
\begin{equation}
\sup _{j\in\mathbb{N}}\| \varphi ^{\varepsilon_i,\delta_j}_0 \|_{C^{2} (\overline{U})} \leq c_1(i )
\label{initialbdd}
\end{equation}
and
\begin{equation}
\lim _{j\to \infty} \| \varphi ^{\varepsilon_i,\delta_j}_0 -\varphi ^{ \varepsilon _i } _0\|_{C^{1,\alpha} (\overline{U})} =0
\label{convphi}
\end{equation}
for $i \geq 1$. Let $\varphi ^{\varepsilon_i,\delta_j}$ be a solution for \eqref{ac} with initial data $\varphi ^{\varepsilon_i,\delta_j}_0$. Then $\sup _{Q_T}|\varphi ^{\varepsilon_i,\delta_j}|\leq \frac{1}{1-\delta_j}$ and $ \sup_{Q_T} |F_{\delta _j} (\varphi^{\varepsilon _i ,\delta_j})| \leq\max_{|s|\leq \frac{1}{1-\delta_j}} |F_{\delta_j} (s) |=1$ by the maximal principle. Thus by \eqref{initialbdd} and the standard arguments for parabolic equations (see \cite[p.517]{ladyzhenskaja}), for any open set $U' \subset\subset U$ there exists $c_2 (i) >0$ such that 
\begin{equation}
\sup _{j\in\mathbb{N}}\| \varphi ^{\varepsilon_i,\delta_j} \|_{C^{1,\alpha} (\overline{Q' _T})} \leq c_2 (i ) ,\ \ i\geq 1,
\label{initialbdd2}
\end{equation}
where $Q'_T :=U' \times (0,T)$. Hence by \eqref{convphi}, \eqref{initialbdd2}, the Arzel\`{a}-Ascoli theorem and the diagonal argument there exists a subsequence $\{ \delta _j\}_{j=1}^\infty$ (denoted by the same index) such that for any compact set $K \subset \mathbb{R}^n$ and $T>0$ we have
\begin{equation}
\varphi ^{\varepsilon_i,\delta_j} \to \varphi ^{\varepsilon_i} \ \ \text{ in }\ C^{1,\alpha} (K \times [0,T]) \ \ \text{and} \ \ \sup _{\mathbb{R}^n \times [0,T]}|\varphi ^{\varepsilon_i}|\leq 1, \ \ i\geq 1,
\label{limphi}
\end{equation}
where $\varphi ^{\varepsilon_i}$ is a solution for \eqref{ac2} with initial data $\varphi^{\varepsilon _i}_0$ (see \cite[Section 2]{chenelliott}). Thus for $ i\geq 1$ and any compact set $K \subset \mathbb{R}^n$ we have
\begin{equation}
e_{i,j} \to e_i \ \ \text{uniformly on} \ K\times [0,T],
\label{uniform}
\end{equation}
where $e_{i,j} =\frac{\varepsilon_{i} |\nabla \varphi ^{\varepsilon_{i},\delta_j}|^2}{2} + \frac{F_{\delta_j} (\varphi ^{\varepsilon_{i},\delta_j} )}{\varepsilon_{i}}$, $e_i = \frac{\varepsilon_{i} |\nabla \varphi ^{\varepsilon_{i}}|^2}{2} + \frac{F_{0} (\varphi ^{\varepsilon_{i}} )}{\varepsilon_{i}}$ and $F_0 (s)= \frac{1-s^2}{2}$. 
Hence
\begin{equation}
\mu ^{\varepsilon_i,\delta_j} _t \to \mu ^{\varepsilon_i}  _t \qquad \text{ as Radon measures}, \ i\geq 1,
\label{uniform2}
\end{equation}
where $\mu ^{\varepsilon_i} _t $ is a Radon measure defined by
\begin{equation}
\mu _t ^{\varepsilon_i}(\phi) := \int _{\mathbb{R}^n} \phi \Big( \frac{\varepsilon_i |\nabla \varphi ^{\varepsilon_i}|^2}{2} + \frac{F_0 (\varphi ^{\varepsilon_i} )}{\varepsilon_i} \Big) dx
\end{equation}
for any $\phi \in C_c (\mathbb{R}^n)$.
By the definition of $\varphi _0 ^{\varepsilon_i ,\delta_j } $ we obtain the following:
\begin{proposition}[see Proposition 1.4 of \cite{Ilmanen}]\label{prop1}
\mbox{}\vspace{-0.175cm}
\begin{enumerate}
\item There exists $D_1 =D_1 (D_0)>0$ such that for any $i,j\geq 1$, we have
\begin{equation}
\sup_{x\in \mathbb{R}^n,R>0} \Big\{ \mu_0 ^{\varepsilon_i ,\delta_j} ( B_R (x)) ,\frac{\mu_0 ^{\varepsilon_i ,\delta_j} ( B_R (x)) }{\omega _{n-1}R^{n-1}} \Big\}\leq D_1.
\label{bound}
\end{equation}
\item $\lim _{i \to \infty}\mu _0 ^{\varepsilon_i} = \frac{\pi}{2} \mathcal{H}^{n-1} \lfloor \Gamma_0$ as Radon measures,
\item $ \lim _{i\to \infty }\varphi _0 ^{\varepsilon _i } =2\chi _{\Omega _0 ^+} -1$ in $BV _{loc}$, 
\item for any $i,j\geq 1$ we have
\begin{equation}
\frac{\varepsilon_i |\nabla \varphi ^{\varepsilon_i ,\delta_j }_0| ^2}{2}\leq \frac{F_{\delta_j} (\varphi ^{\varepsilon_i ,\delta_j }_0)}{\varepsilon_i } \quad \text{on} \quad \mathbb{R}^n.
\label{negativity}
\end{equation}
\end{enumerate}
\end{proposition}
\begin{proof}
We only prove (2) and (4). In the same manner as \cite{Ilmanen} we have
\[ \lim _{i \to \infty}\mu _0 ^{\varepsilon_i} = \lim _{\delta \downarrow 0} \sigma _\delta \mathcal{H}^{n-1} \lfloor \Gamma_0. \]
By $ \lim _{\delta \downarrow 0} \sigma _\delta= \int _{-1} ^{1} \sqrt{2F_0 (s)} \, ds=\frac{\pi}{2}$ we obtain (2). We compute that
\[
\frac{\varepsilon_i |\nabla \varphi _0 ^{\varepsilon_i,\delta_j} |^2/2}{F_{\delta_j}(\varphi_0 ^{\varepsilon_i,\delta_j}) / \varepsilon_i}= 
\frac{\varepsilon_i ( q^{\varepsilon_i,\delta_j} _r)^2 /2}{F_{\delta_j} (q ^{\varepsilon_i,\delta_j})/ \varepsilon_i}|\nabla\overline{ r_{\varepsilon_i} }|^2
= |\nabla \overline{r_{\varepsilon_i}}|^2\leq 1 ,
\]
where \eqref{q} and $ |\nabla \overline{r_{\varepsilon_i}}|\leq 1$ are used. Hence we obtain \eqref{negativity}.
\end{proof}
%
%
%
%
Our main results are the following:
\begin{theorem}\label{main}
Let $\Omega ^+ _0 \subset \mathbb{R}^n$ be a bounded open set and satisfy \eqref{initialdata1} and \eqref{initialdata2}. Let $\varphi ^{\varepsilon _i } \in C^{1,\alpha}_{loc}(\mathbb{R}^n \times (0,\infty))$ be a solution for \eqref{ac2} with initial data $\varphi ^{\varepsilon _i} _0$, and $\varphi ^{\varepsilon _i ,\delta_j} \in C^{2,\alpha}_{loc}(\mathbb{R}^n \times (0,\infty))$ be a solution for \eqref{ac} with initial data $\varphi ^{\varepsilon _i ,\delta_j} _0$, where  $\varphi ^{\varepsilon _i} _0$ and $\varphi ^{\varepsilon _i ,\delta_j} _0$ are defined by \eqref{initial}. 

Then there exist 
\begin{enumerate}
\item[(a)] subsequences $ \{ i_k \}_{k=1}^\infty$, $ \{ j_k \}_{k=1}^\infty$ and a family of Radon measures $\{ \mu _t\} _{t\geq 0}$ such that

\begin{equation}
\mu_t ^{\varepsilon_{i},\delta_{j_k}} \to \mu_t ^{\varepsilon_{i}} \ \ \text{as} \ \ k \to \infty, \ \ t\geq 0, \ i\geq 1,
\label{result1}
\end{equation}

\begin{equation}
\mu_t ^{\varepsilon_{i_k},\delta_{j_k}} \to \mu_t \ \ \text{as} \ \ k \to \infty, \ \ t\geq 0,
\label{result2}
\end{equation}

\begin{equation}
\mu_t ^{\varepsilon_{i_k}} \to \mu_t \ \ \text{as} \ \ k \to \infty, \ \ t\geq 0
\label{result3}
\end{equation}
and $\{ \mu_t\}_{t\geq 0} $ is a global solution for Brakke's mean curvature flow with initial data $\mu_0 =\frac{\pi}{2} \mathcal{H}^{n-1} \lfloor \Gamma _0$,
\item[(b)] 
and $\varphi \in BV _{loc}( \mathbb{R}^n \times [0,\infty)) \cap C^{\frac{1}{2}} _{loc} ([0,\infty) ;L^1 (\mathbb{R}^n))$ such that
\begin{enumerate}
\item[(b1)] $\varphi ^{\varepsilon_{i_k},\delta_{j_k}} \to 2\varphi -1 \quad \text{in} \ L^1 _{loc} ( \mathbb{R}^n \times [0,\infty))$ and a.e. pointwise,
\item[(b2)] $\varphi (\cdot,0) =\chi _{\Omega_0 ^+} $ a.e. on $\mathbb{R}^n$,
\item[(b3)] $\varphi (\cdot,t)$ is a characteristic function for all $t\in [0,\infty)$,
\item[(b4)] $\| \nabla \varphi (\cdot,t) \| (\phi) \leq \frac{2}{\pi} \mu _t (\phi) $ for any $t\in [0,\infty)$ and $\phi \in C_c (\mathbb{R}^n;\mathbb{R}^+)$. Moreover $\spt\| \nabla \varphi (\cdot,t) \| \subset \spt \mu _t  $ for any $t\in [0,\infty)$.
\end{enumerate}
\end{enumerate}
\end{theorem}


\section{Monotonicity formula}In this section, we consider the monotonicity formula for $\mu _t ^{\varepsilon_i, \delta_j}$ and prove the negativity of the discrepancy measure which we define later. We assume that $\Omega ^+ _0 \subset \mathbb{R}^n$ is a bounded open set and satisfies \eqref{initialdata1} and \eqref{initialdata2}, and $\varphi ^{\varepsilon_i, \delta_j} \in C^{2,\alpha}_{loc}(\mathbb{R}^n\times (0,\infty))$ is a solution for \eqref{ac} with initial data $\varphi ^{\varepsilon_i, \delta_j} _0$, where  $\varphi ^{\varepsilon_i, \delta_j} _0$ is defined by \eqref{initial} in this section. We denote $\varepsilon _i$ and $\delta_j$ by $\varepsilon$ and $\delta$.

\bigskip

We define the backward heat kernel $\rho$ by
\[ \rho= \rho _{y,s} (x,t) := \frac{1}{(4\pi (s-t))^{\frac{n-1}{2}}} e^{-\frac{|x-y|^2}{4(s-t)}}, \qquad t<s, \ x,y\in \mathbb{R}^n. \]
We define a Radon measure $\xi _t ^{\varepsilon,\delta}$ by
\begin{equation}
\xi _t ^{\varepsilon,\delta}(\phi) := \int _{\mathbb{R}^n} \phi \Big( \frac{\varepsilon |\nabla \varphi ^{\varepsilon,\delta}|^2}{2} - \frac{F_\delta (\varphi ^{\varepsilon,\delta} )}{\varepsilon} \Big) dx,
\end{equation}
for any $\phi \in C_c (\mathbb{R}^n)$. $\xi _t ^{\varepsilon,\delta}$ is called a discrepancy measure. The monotonicity formula for the mean curvature flow is proved by Huisken~\cite{huisken1990}. Ilmanen~\cite{Ilmanen} proved the monotonicity formula for the Allen-Cahn equation without constraint. The following monotonicity formula is obtained in the same manner as \cite[3.3]{Ilmanen}. So we skip the proof.
\begin{proposition}\label{mono1}
\begin{equation}
\begin{split}
\frac{d}{dt} \int _{\mathbb{R}^n} \rho \, d\mu _t ^{\varepsilon,\delta}(x) =& -\int _{\mathbb{R}^n} \varepsilon \rho \Big( -\Delta \varphi ^{\varepsilon,\delta} +\frac{F' _{\delta} (\varphi ^{\varepsilon,\delta}) }{\varepsilon ^2} -\frac{\nabla \varphi ^{\varepsilon,\delta} \cdot \nabla \rho}{\rho} \Big)^2 \, d\mu _t^{\varepsilon,\delta}(x) \\
&+ \frac{1}{2(s-t)}\int _{\mathbb{R}^n}  \rho \, d\xi _t ^{\varepsilon,\delta}(x) 
\end{split}
\label{monoton}
\end{equation}
for $y\in \mathbb{R}^n$ and $0\leq t<s$.
\end{proposition}
Define $\xi _{\varepsilon,\delta} =\xi _{\varepsilon,\delta}(x,t):=\frac{\varepsilon |\nabla \varphi ^{\varepsilon,\delta}(x,t)|^2}{2} - \frac{F_\delta (\varphi ^{\varepsilon,\delta}(x,t) )}{\varepsilon}$.
\begin{proposition}\label{negative}
$\xi _{\varepsilon,\delta} (x,t)\leq 0$ for any $(x,t) \in \mathbb{R}^n \times [0,\infty)$. Moreover $\xi _t ^{\varepsilon,\delta} $ is a non-positive measure for $t \in [0,\infty)$.
\end{proposition}
\begin{proof}
Let $h>0$ and $F_{\delta,h}\in C^\infty (\mathbb{R})$ be a function with $\lim _{h\to 0} \| F_{\delta ,h} -F_\delta\| _{C^1 (\mathbb{R})} =0$. 

Let $q^{\varepsilon,\delta,h}\in C^{\infty}(\mathbb{R})$ be a solution for
\begin{equation}
\frac{\varepsilon ( q^{\varepsilon,\delta,h} _r)^2 }{2} = \frac{F_{\delta,h} (q ^{\varepsilon,\delta,h})}{\varepsilon} \quad \text{on} \ \mathbb{R}
\label{qh}
\end{equation}
with $\lim_{h\to 0} \| q^{\varepsilon,\delta}-q^{\varepsilon,\delta,h} \|_{C^{2}([-L,L])}=0$ for any $L>0$. We remark that we have
\begin{equation} 
q^{\varepsilon,\delta,h} _{rr}= \frac{F' _{\delta,h} (q ^{\varepsilon,\delta,h})}{\varepsilon ^2} \quad \text{on} \ \mathbb{R}.
\label{qh2}
\end{equation}

Let $\varphi^{\varepsilon,\delta,h}\in C^{2,\alpha} (\mathbb{R}\times (0,\infty))$ be a solution for
\begin{equation}
\left\{ 
\begin{array}{ll}
\partial _t \varphi ^{\varepsilon,\delta,h} -\Delta \varphi ^{\varepsilon,\delta,h} +\dfrac{F'_{\delta,h} (\varphi ^{\varepsilon,\delta,h})}{\varepsilon ^2}=0,& (x,t)\in \mathbb{R}^n \times (0,\infty),  \\
\varphi ^{\varepsilon,\delta,h} (x,0) = \varphi _0 ^{\varepsilon,\delta,h} (x) ,  &x\in \mathbb{R}^n,
\end{array} \right.
\label{ac3}
\end{equation}
where $\varphi _0 ^{\varepsilon,\delta,h}$ is defined by
\begin{equation*}
\varphi ^{\varepsilon,\delta,h}_0(x)=q^{\varepsilon, \delta,h} (\overline{r_\varepsilon} (x)), \quad x\in \mathbb{R}^n.
\end{equation*}
We define a function $r:\mathbb{R}^n \times [0,\infty) \to \mathbb{R}$ by
\[ \varphi ^{\varepsilon,\delta,h} (x,t) =q^{\varepsilon,\delta,h} (r(x,t)) , \quad (x,t)\in \mathbb{R}^n \times [0,\infty).\]
By \eqref{qh} we have
\[ \frac{\varepsilon |\nabla \varphi ^{\varepsilon,\delta,h} |^2/2}{F_{\delta,h}(\varphi ^{\varepsilon,\delta,h}) / \varepsilon}\leq |\nabla r|^2 \ \ \text{on} \ \ \mathbb{R}^n \times[0,\infty).\]
Hence, if  $|\nabla r|\leq 1$ for any $h>0$ then $\displaystyle \frac{\varepsilon |\nabla \varphi ^{\varepsilon,\delta} |^2 /2}{F_{\delta}(\varphi ^{\varepsilon,\delta}) / \varepsilon}\leq 1$.
Thus we only need to prove that $|\nabla r|\leq 1$ on $\mathbb{R}^n \times [0,\infty)$. 

Let $g^{\delta,h} (s):=\sqrt{2F_{\delta,h} (s)}$. By \eqref{qh} and \eqref{qh2} we have
\begin{equation}
q^{\varepsilon,\delta,h} _r =\frac{ g^{\delta,h} (q^{\varepsilon,\delta,h})}{\varepsilon} \qquad \text{and} \qquad q^{\varepsilon,\delta,h} _{rr}= \frac{( g^{\delta,h} (q^{\varepsilon,\delta,h}) )_r}{\varepsilon}=\frac{g^{\delta,h} _q ( q^{\varepsilon,\delta,h} ) }{\varepsilon} q^{\varepsilon,\delta,h} _r.
\label{q2} 
\end{equation}
By \eqref{qh2}, \eqref{ac3} and \eqref{q2} we obtain
\begin{equation}
\begin{split}
 q_r ^{\varepsilon,\delta,h} \partial _t r &= q_r ^{\varepsilon,\delta,h} \Delta r + q_{rr} ^{\varepsilon,\delta,h} |\nabla r|^2 -q_{rr} ^{\varepsilon,\delta,h}\\
&=  q_r ^{\varepsilon,\delta,h} \Delta r + q_{r} ^{\varepsilon,\delta,h} \frac{g^{\delta,h} _q}{\varepsilon} (|\nabla r|^2 -1).
\end{split}
\end{equation}
Thus we have
\[ \partial _t r =  \Delta r + \frac{g ^{\delta,h} _q}{\varepsilon} (|\nabla r|^2 -1) \]
and
\begin{equation}
\partial _t |\nabla r|^2=\frac{1}{2}\Delta |\nabla r |^2 - |\nabla^2 r|^2 +\frac{2}{\varepsilon} \nabla r\cdot \nabla g^{\delta,h} _q (|\nabla r|^2 -1) + \frac{2}{\varepsilon} \nabla r\cdot \nabla |\nabla r|^2. 
\label{max}
\end{equation}
By the assumption we have $|\nabla r(\cdot,0)|=|\nabla \overline{r_\varepsilon}|\leq 1$ on $\mathbb{R}^n$. By \eqref{max} and the maximal principle we obtain $|\nabla r|\leq 1$ in $\mathbb{R}^n \times [0,\infty)$.
\end{proof}
By Proposition $\ref{mono1}$ and Proposition $\ref{negative}$ we have
\begin{proposition}For $y\in \mathbb{R}^n$ and $0\leq t<s$ we have
\begin{equation}
\begin{split}
\frac{d}{dt} \int _{\mathbb{R}^n} \rho \, d\mu _t ^{\varepsilon,\delta}(x) \leq -\int _{\mathbb{R}^n} \varepsilon \rho \Big( -\Delta \varphi ^{\varepsilon,\delta} +\frac{F'_{\delta} (\varphi ^{\varepsilon,\delta}) }{\varepsilon ^2} -\frac{\nabla \varphi ^{\varepsilon,\delta} \cdot \nabla \rho}{\rho} \Big)^2 \, d\mu _t^{\varepsilon,\delta} \leq 0.
\end{split}
\label{monoton2}
\end{equation}

\end{proposition}\label{densityprop}
Next we prove the upper density ratio bounds of $\mu _t ^{\varepsilon,\delta}$.
\begin{proposition}
There exists $c_3=c_3 (n)>0$ such that
\begin{equation}
\mu _t ^{\varepsilon,\delta} (B_R(x)) \leq c_3 D_1 R^{n-1}
\label{density}
\end{equation}
for $(x,t) \in \mathbb{R}^n \times [0,\infty)$ and $R>0$.
\end{proposition}
\begin{proof}
We compute that
\begin{equation}
\begin{split}
&\int _{\mathbb{R}^n}\rho _{y,s}(x,0) \, d\mu _0 ^{\varepsilon ,\delta}(x) =\frac{1}{(4\pi s)^{\frac{n-1}{2}} } \int _{\mathbb{R}^n} e^{-\frac{|x-y|^2}{4s}} \, d\mu _0 ^{\varepsilon ,\delta} \\
=& \frac{1}{(4\pi s)^{\frac{n-1}{2}} } \int _0 ^1 \mu _0 ^{\varepsilon ,\delta}( \{ x \, | \, e^{-\frac{|x-y|^2}{4s}}>k \}) \, dk 
= \frac{1}{(4\pi s)^{\frac{n-1}{2}} } \int _0 ^1 \mu _0 ^{\varepsilon ,\delta}( B_{\sqrt{4s\log k^{-1}}} (y)) \, dk\\
\leq &  \frac{1}{(4\pi s)^{\frac{n-1}{2}} } \int _0 ^1 D_1 \omega_{n-1} (\sqrt{4s\log k^{-1}})^{n-1}  \, dk \leq c_4 D_1,
\end{split}
\end{equation}
where $c_4 >0$ is depending only on $n$ and the density upper bound \eqref{bound} is used. By the monotonicity formula \eqref{monoton2}, we have
\begin{equation}
\begin{split}
&\int _{\mathbb{R}^n}\rho _{y,s}(x,t) \, d\mu _t ^{\varepsilon ,\delta}(x)\leq \int _{\mathbb{R}^n}\rho _{y,s}(x,0) \, d\mu _0 ^{\varepsilon ,\delta}(x) \leq c_4 D_1,
\end{split}
\label{density1}
\end{equation}
for any $0<t<s$ and $y\in\mathbb{R}^n$. Fix $R>0$ and set $s=t+\frac{R^2}{4}$. Then
\begin{equation}
\begin{split}
&\int _{\mathbb{R}^n}\rho _{y,s}(x,t) \, d\mu _t ^{\varepsilon ,\delta}=\int _{\mathbb{R}^n} \frac{1}{\pi ^{\frac{n-1}{2}} R^{n-1}} e^{-\frac{|x-y|^2}{R^2}} \, d \mu _t ^{\varepsilon ,\delta}\geq \int _{B_R (y)} \frac{1}{\pi ^{\frac{n-1}{2}} R^{n-1}} e^{-\frac{|x-y|^2}{R^2}} \, d \mu _t ^{\varepsilon ,\delta}\\
\geq & \int _{B_R (y)} \frac{1}{\pi ^{\frac{n-1}{2}} R^{n-1}} e^{-1} \, d \mu _t ^{\varepsilon ,\delta}=\frac{1}{e \pi ^{\frac{n-1}{2}} R^{n-1}}\mu _t^{\varepsilon,\delta} (B_R (y)).
\end{split}
\label{density2}
\end{equation}
By \eqref{density1} and \eqref{density2} we obtain \eqref{density}.
\end{proof}

\section{Existence of limit measures}
In this section, we prove the existence of limit measure $\mu _t $. We also assume that $\Omega ^+ _0 \subset \mathbb{R}^n$ satisfies \eqref{initialdata1} and \eqref{initialdata2}, $\varphi ^{\varepsilon_i, \delta_j} \in C^{2,\alpha}_{loc}(\mathbb{R}^n\times (0,\infty))$ is a solution for \eqref{ac} with initial data $\varphi ^{\varepsilon_i, \delta_j} _0$, where  $\varphi ^{\varepsilon_i, \delta_j} _0$ is defined by \eqref{initial}, and \eqref{limphi} and \eqref{uniform2} hold in this section.

\begin{lemma}\label{noninc}
For any $\phi \in C_c ^2 (\mathbb{R}^n;\mathbb{R}^+)$, $i,j\geq 1$ and $t>0$ we have
\begin{equation}
\frac{d}{dt}\mu _t ^{\varepsilon_i,\delta_j}(\phi)\leq \sup _{x\in \mathbb{R}^n} |\nabla ^2 \phi| \mu _t  ^{\varepsilon_i,\delta_j}(\spt \phi).
\label{mono}
\end{equation}
Moreover there exists $c_5=c_5 (n,D_1,\spt \phi, \sup _{x\in \mathbb{R}^n} |\nabla ^2 \phi|)>0$ such that the function $\mu _t ^{\varepsilon_i}(\phi)-c_5 t$ of $t$ is nonincreasing for any $i\geq 1$.
\end{lemma}
\begin{proof}
We denote $\varepsilon _i$, $\delta_j$ and $\varphi ^{\varepsilon _i ,\delta _j}$ by $\varepsilon$, $\delta$ and $\varphi$. By the integration by parts,
\begin{equation}
\begin{split}
&\frac{d}{dt}\mu _t ^{\varepsilon,\delta}(\phi) = \int _{\mathbb{R}^n} \phi \frac{\partial }{\partial t} \Big( \frac{\varepsilon |\nabla \varphi| ^2}{2} +\frac{F_\delta (\varphi )}{\varepsilon} \Big) \, dx \\
=&  \int  _{\mathbb{R}^n} \phi \Big( \varepsilon \nabla \varphi \cdot \nabla \varphi _t +\frac{F' _\delta (\varphi )}{\varepsilon}\varphi _t \Big) \, dx \\
=&  \int_{\mathbb{R}^n}  \varepsilon \phi \Big( -\Delta \varphi  +\frac{F' _\delta (\varphi )}{\varepsilon ^2} \Big)\varphi _t -\varepsilon (\nabla \phi \cdot \nabla\varphi ) \varphi _t \, dx\\
=&  \int_{\mathbb{R}^n}  -\varepsilon \phi \Big( -\Delta \varphi  +\frac{F' _\delta (\varphi )}{\varepsilon ^2} \Big) ^2  + \varepsilon (\nabla \phi \cdot \nabla\varphi ) \Big( -\Delta \varphi  +\frac{F' _\delta(\varphi )}{\varepsilon ^2} \Big) \, dx\\
=&  \int_{\mathbb{R}^n}  -\varepsilon \phi \Big( -\Delta \varphi  +\frac{F' _\delta (\varphi )}{\varepsilon ^2} -\frac{\nabla \phi \cdot \nabla \varphi }{2\phi}\Big) ^2  + \varepsilon \frac{(\nabla \phi \cdot \nabla \varphi )^2}{4\phi} \, dx\\
\leq &   \Big( \sup _{x\in \{ x \, | \, \phi(x)>0 \} } \frac{|\nabla \phi|^2}{2\phi}\Big) \mu ^{\varepsilon,\delta} (\spt \phi ) \leq \sup _{x\in \mathbb{R}^n} |\nabla ^2 \phi|\mu ^{\varepsilon,\delta} (\spt \phi ),
\end{split}
\label{deri}
\end{equation}
where $\sup _{x\in \{ x \, | \, \phi(x)>0 \} } \frac{|\nabla \phi|^2}{2\phi}\leq \sup _{x\in \mathbb{R}^n} |\nabla ^2 \phi|$ are used. By \eqref{density} and \eqref{deri} there exists $c_5=c_5 (n,D_1,\spt \phi, \sup _{x\in \mathbb{R}^n} |\nabla ^2 \phi|)>0$ such that  $\mu _t ^{\varepsilon,\delta}(\phi)-c_5 t$ of $t$ is nonincreasing. By $\mu _t^{\varepsilon,\delta} \to \mu _t ^{\varepsilon}$ for any $\varepsilon >0$, $\mu _t ^{\varepsilon}(\phi)-c_5 t$ of $t$ is also nonincreasing.
\end{proof}
\begin{remark}
By an argument similar to that in the proof of Lemma \ref{noninc}, for any $i,j \geq 1$ and $T\geq 0$ we have
\begin{equation}
\mu _T ^{\varepsilon_i,\delta_j}(\mathbb{R}^n)+ \int _0 ^T \int_{\mathbb{R}^n}  \varepsilon _i \Big( -\Delta \varphi^{\varepsilon _i ,\delta _j}  +\frac{F' _{\delta_j} (\varphi^{\varepsilon _i ,\delta _j} )}{\varepsilon_i ^2} \Big) ^2 \, dxdt = \mu _0 ^{\varepsilon_i,\delta_j}(\mathbb{R}^n) \leq D_1.
\label{boundrn}
\end{equation}
\end{remark}

\begin{proposition}\label{existencemeasure}
There exist subsequences $ \{\varepsilon_{ i_k} \}_{k=1}^\infty, \ \{ \delta_{ j_k } \}_{k=1}^\infty$ and a family of Radon measures $\{ \mu _t\} _{t\geq 0}$ such that \eqref{result1}, \eqref{result2} and \eqref{result3} hold.
\end{proposition}

\begin{proof}
By \eqref{uniform2} we only need to prove \eqref{result2} and \eqref{result3}. First we prove that there exist $ \{\varepsilon_{ i_k} \}_{k=1}^\infty$ and a family of Radon measures $\{ \mu _t\} _{t\geq 0}$ such that $\mu_t ^{\varepsilon_{i_k}} \to \mu_t$ as $ k\to \infty$ for any $t\geq 0$.

Let $B_1 \subset [0,\infty)$ be a countable dense set. By the compactness of Radon measures, there exists a subsequence  $\{\varepsilon_{i_k} \}_{k = 1}^\infty$ and a family of Radon measures $\{  \mu _t \}_{t\in B_1}$ such that $\mu _t ^{\varepsilon _{i_k}} \to \mu _t $ as $k \to \infty$ for any $t\in B_1$. Let $\{ \phi _l \}_{l= 1}^\infty \subset C_c ^2 (\mathbb{R}^n ;\mathbb{R}^+)$ be a dense set. 

By Lemma $\ref{noninc}$, $\mu_t (\phi _l) -c_5 (\phi _l) t$ of $t\in B_1$ is nonincreasing for any $l \geq 1$. Hence for any $l \geq 1$ there exists a countable set $E_l \subset [0,\infty)$ such that 
\begin{equation}
\lim _{t \uparrow s, t \in B_1} \mu _t (\phi _l) = \lim _{t\downarrow  s, t\in B_1} \mu _t (\phi _l) 
\label{1}
\end{equation}
for any $s\in [0,\infty )\setminus E_l$. Set $B_2 = [0,\infty) \setminus \cup _l E_l$. Then $B_2$ is co-countable and \eqref{1} holds for any $l\geq 1$ and $s\in B_2$.

Let $s\in B_2 \setminus B_1$. By the compactness of Radon measures, there exist a subsequence  $\{\varepsilon_{ i _{ k _m} } \}_{m = 1}^\infty$ and a Radon measure $ \mu _s$ such that $\mu _s ^{\varepsilon_{i_{k_m} }} \to \mu _s $ as $m \to \infty$. 

Next we show that $\mu _s$ is unique and $\mu _s ^{\varepsilon _{i_k}} \to \mu _s $ as $k \to \infty$. By Lemma $\ref{noninc}$, for any $l \geq 1$, $m \in \mathbb{N}$ and $t_1, t_2$ with $t_1<s<t_2$ we have
\[ \mu _{t_1} ^{\varepsilon_{i_{k_m} } }(\phi _l) -c_5 (t_1-s) \geq \mu _{s} ^{\varepsilon_{ i_{k_m} }} (\phi _l) \geq \mu _{t_2} ^{\varepsilon_{ i_{k_m}}} (\phi _l)-c_5 (t_2-s). \]
Hence for any $t_1 ,t_2 \in B_1$ with $t_1<s<t_2$ we have
\[ \mu _{t_1} (\phi _l) -c_5 (t_1-s) \geq \mu _{s} (\phi _l) \geq \mu _{t_2} (\phi _l)-c_5 (t_2-s). \]
Therefore by \eqref{1} we obtain $\mu _{s} (\phi _l) =\lim _{t \uparrow s, t \in B_1} \mu _t (\phi _l) = \lim _{t\downarrow  s, t\in B_1} \mu _t (\phi _l ) $ for any $l\geq 1$. Thus $\mu _s$ is uniquely determined. Moreover, $\mu _s ^{\varepsilon_{i_k}} \to \mu _s $ as $k \to \infty$.

Therefore $\mu _t ^{\varepsilon_{i_k}} \to \mu _t $ as $k\to \infty$ for any $t\in B_1 \cup B_2$. Because $[0,\infty) \setminus (B_1 \cup B_2)$ is a countable set, there exists a subsequence $\{ \varepsilon_{i _ k} \}_{k = 1}^\infty$ (denoted by the same index) such that $\mu _t ^{\varepsilon_{i _ k}} \to \mu _t $ as $k\to \infty$ for any $t\in [0,\infty)$.

Next we show that there exists a subsequences $ \{ \delta_{ j_k } \}_{k=1}^\infty$ such that $\mu_t ^{\varepsilon_{i_k},\delta_{j_k}} \to \mu_t$ as $ k\to \infty$ for any $t\geq 0$. For $\phi \in C_c (\mathbb{R}^n)$ we compute that

\begin{equation}
\begin{split}
|\mu_t ^{\varepsilon_{i_k},\delta_{j}} (\phi ) - \mu_t(\phi )|&\leq |\mu_t ^{\varepsilon_{i_k},\delta_{j}} (\phi ) - \mu_t ^{\varepsilon_{i_k}}(\phi )|
+ |\mu_t ^{\varepsilon_{i_k}} (\phi ) - \mu_t(\phi)|\\
&\leq \sup_{\mathbb{R}^n}|\phi| \int _{\spt \phi} |e_{i_k,j}(x,t) -e_{i_k}(x,t)| \, dx + |\mu_t ^{\varepsilon_{i_k}} (\phi ) - \mu_t(\phi)|.
\end{split}
\label{tele}
\end{equation}
Let $\{ R_k \} _{k=1}^\infty $ and $\{ T_k \} _{k=1} ^\infty$ satisfy $R_k ,T_k \to \infty$ as $k\to \infty$. By \eqref{uniform} and the diagonal argument, there exists a subsequence $\{\delta _{j_k} \}_{k=1} ^\infty$ such that
\begin{equation}
\sup_{t\in [0,T_k]} \int _{B_{R_k} (0)} |e_{i_k,j_k}(x,t) -e_{i_k}(x,t)| \, dx \leq \frac{1}{k} \quad \text{for} \quad k\geq 1. 
\label{tele2}
\end{equation} 
By \eqref{tele} and \eqref{tele2} we have $\mu_t ^{\varepsilon_{i_k},\delta_{j_k}} \to \mu_t $ as $ k\to \infty$ for any $t\geq 0$. Hence we obtain \eqref{result1}, \eqref{result2} and \eqref{result3}.
\end{proof}

\section{Forward density lower bound and vanishing of $\xi$}
In this section we prove the lower density estimate for $\mu _t$ and the vanishing of $\xi$ by using the technique of Ilmanen~\cite{Ilmanen} and Takasao and Tonegawa~\cite{takasaotonegawa}. Assume that $\varphi^{\varepsilon_i,\delta_j}$ and $\mu ^{\varepsilon_i,\delta_j} _t$ satisfy all the assumptions of Section 2 and $\mu ^i _t :=\mu ^{\varepsilon _i,\delta_i} _t \to \mu _t$ and $\xi ^i _t :=\xi ^{\varepsilon _i,\delta_i} _t \to \xi _t$ for any $t\geq 0$ as Radon measures in this section. We denote $\varphi^i :=\varphi ^{\varepsilon _i,\delta _i}$ in this section.

\bigskip

By the computation we have the following estimates. The proof is omitted.
\begin{lemma}\label{estimatef}
We have
\begin{equation}
F' _\delta (s)\Big(s-\frac{1}{1-\delta}\Big) \geq \frac{1}{100} F_\delta (s)
\end{equation}
for any $\delta \in (0,\frac{3}{10})$ and $s \in [\frac{3}{4}, \infty)$, and
\begin{equation}
F' _\delta (s)\Big(s+\frac{1}{1-\delta}\Big) \geq \frac{1}{100} F_\delta (s)
\end{equation}
for any $\delta \in (0,\frac{3}{10})$ and $s \in (-\infty ,-\frac{3}{4}]$.
\end{lemma}
Let $\mu$ be a measure on $\mathbb{R}^n \times [0,\infty)$ such that $d\mu =d\mu _t dt$. 
\begin{lemma}\label{lemspt}
Assume $(x',t')\in \spt \mu$. Then there exist a subsequence $\{ i_j \}_{j=1} ^\infty$ and  $\{ (x_j ,t_j) \}_{j=1}^\infty $ such that 
\begin{equation}
\lim _{j\to \infty} (x_j , t_j)= (x',t') \qquad \text{and} \qquad |\varphi ^{i_j} (x_j, t_j)|<\frac{3}{4} 
\end{equation}
for any $j\in\mathbb{N}$.
\end{lemma}

\begin{proof}Set $Q_r = B_r (x') \times [t'-r^2,t'+r^2]$ for $r>0$. If the claim were not true, then there exist $r>0$ and $N\in \mathbb{N}$ such that $\inf _{Q_r} |\varphi ^i | \geq \frac{3}{4}$ for any $i\geq N$. So we may assume that $\inf _{Q_r} \varphi ^i  \geq \frac{3}{4}$ for any $i\geq N$ without loss of generality. Moreover we may assume $\delta _i \in (0,\frac{3}{10})$ for $i\geq N$. Let $\phi\in C_c^2 (Q_r)$. Then by Lemma \ref{estimatef} and $\sup_{Q_r } |\varphi ^i| \leq \frac{1}{1-\delta ^i}\leq 2$ we have
\begin{equation}
\begin{split}
&\frac{1}{100} \int _{Q_r} \phi ^2 \frac{F_{\delta _i} (\varphi ^i)}{\varepsilon _i ^2}\, dxdt \leq \int _{Q_r} \phi ^2 \frac{F' _{\delta _i} (\varphi ^i)}{\varepsilon _i ^2} \Big(\varphi ^i -\frac{1}{1-\delta _i}\Big) \, dxdt \\
\leq& \int _{Q_r} \phi ^2 (-\varphi ^i _t +\Delta \varphi ^i) (\varphi ^i -\frac{1}{1-\delta _i}) \, dxdt \\
=& \int _{t'-r^2}^{t+r^2}\frac{d}{dt} \Big( \int_{B_r(x')} \phi ^2 \Big(-\frac{1}{2} (\varphi ^i )^2 +\frac{1}{1-\delta _i} \varphi ^i\Big) \, dx\Big) dt \\
&+\int _{Q_r} \frac{2}{1-\delta _i} \phi \nabla \phi \cdot \nabla\varphi ^i -\phi ^2 |\nabla \varphi ^i|^2 -2 \phi \varphi ^i \nabla \phi \cdot \nabla \varphi ^i \, dxdt\\
\leq & C (\phi) + \int _{Q_r}  -\phi ^2 |\nabla \varphi ^i|^2 +\frac{1}{2} \phi ^2 |\nabla \varphi ^i|^2 +4|\nabla \phi|^2 \, dxdt \leq C (\phi),
\end{split}
\label{51}
\end{equation}
where $C(\phi) >0$ depends only on $\sup _{x\in \mathbb{R}^n} \{ |\phi|,|\nabla \phi| \}$. By Proposition \ref{negative} and \eqref{51} we obtain
\begin{equation*}
\begin{split}
\int _{t'-r^2} ^{t'+r^2}  \int _{B_r (x')} \phi ^2 \, d\mu _t ^{i} dt \leq 2\int _{Q_r} \phi ^2 \frac{F_{\delta _i} (\varphi ^i)}{\varepsilon _i }\, dxdt \leq 200 C (\phi) \varepsilon _i.
\end{split}
\end{equation*}
Hence we have
\[ \int _{Q_r } \phi ^2 \, d\mu  =0. \]
This proves that $(x',t') \not \in \spt \mu$.
\end{proof}
Set
\[ \rho _{y} ^r (x) := \frac{1}{( \sqrt{2 \pi } r ) ^{n-1} }e^{-\frac{|x-y|^2}{2 r^2}}, \qquad r>0, \ x,y\in \mathbb{R}^n. \]
Note that $\rho _{y,s}(x,t)=\rho _y ^r (x)$ for $r=\sqrt{2(s-t)}$. We use the following estimates.
\begin{lemma}[See\cite{Ilmanen}]\label{estilmanen}
Let $D>0$ and $\nu$ be a measure satisfying $\sup _{R>0 ,x \in \mathbb{R}^n} \frac{ \nu (B_R(x)) }{ \omega _{n-1} R^{n-1} } \leq D$. Then the following hold:
\begin{enumerate}
\item For any $a>0$ there is $\gamma_1 =\gamma_1 (a)>0$ such that for any $r>0$ and $x,x_1 \in \mathbb{R}^n$ with $|x-x_1|\leq \gamma_1 r$ we have
\begin{equation}
\int_{\mathbb{R}^n} \rho_{x_1 } ^r (y) \,d\nu (y) \leq (1+a) \int _{\mathbb{R}^n} \rho_{x} ^r \,d\nu (y) +aD.
\label{rhoest1}
\end{equation} 
\item For any $r,R>0$ and $x\in \mathbb{R}^n$ we have
\begin{equation}
\int_{\mathbb{R}^n \setminus B_R(x)} \rho^r _x (y) \,d\nu (y) \leq 2^{n-1} e^{-3R^2 /8r^2}D.
\label{rhoest2}
\end{equation} 
\item For any $a>0$ there is $\gamma_2 =\gamma_2 (a)>0$ such that for any $r,R>0$ with $1\leq \frac{R}{r} \leq 1+\gamma _2$ and any $x\in \mathbb{R}^n$, we have
\begin{equation}
\int_{\mathbb{R}^n} \rho_{x}^R (y) \,d\nu (y) \leq (1+a) \int _{\mathbb{R}^n} \rho_x ^r (y) \,d\nu (y) +aD.
\label{rhoest3}
\end{equation} 
\end{enumerate}
\end{lemma}

\begin{lemma}\label{lemeta}
There exist $\eta=\eta (n)>0$ and $\gamma _3=\gamma _3(n,D_1) >0$ with the following property. Given $0\leq t<s$, define $r=\sqrt{2(s-t)}$ and $t'=s+\frac{r^2}{2}$. If $x\in \mathbb{R}^n$ satisfies
\begin{equation}
\int _{\mathbb{R}^n} \rho _{y,s} (x,t) \, d\mu _s (y) < \eta,
\label{assumption1}
\end{equation}
then $(\bar B_{\gamma _3 r} (x) \times \{ t' \}) \cap \spt \mu=\emptyset$.
\end{lemma}
\begin{proof}
Assume for a contradiction that $(x',t')\in \spt \mu$ for some $x' \in \bar B_{\gamma _3 r} (x)$ with \eqref{assumption1}, where $\gamma _3$ will be chosen later. Then by Lemma \ref{lemspt} there exist a sequence $\{ (x_j,t_j) \}_{j=1} ^{\infty}$ and $\{ i_j \}_{j=1} ^\infty$ such that $\lim _{j\to \infty} (x_j,t_j) =(x',t')$ and $|\varphi ^{i_j} (x_j ,t_j) | <\frac{3}{4}$ for any $j$. By Proposition \ref{negative} and $\sup _{\mathbb{R}^n \times [0,\infty) }|\varphi ^{i}  | \leq \frac{1}{1-\delta _i}$ for any $i\geq 1$, we have 
\begin{equation}
\sup _{\mathbb{R}^n \times [0,\infty) } |\nabla\varphi ^{i}| \leq \sup _{\mathbb{R}^n \times [0,\infty) } \frac{\sqrt{2F_{\delta _i} (\varphi ^i) }}{\varepsilon _i} \leq \frac{1}{\varepsilon _i}\quad \text{for} \quad i\geq 1.
\end{equation}
Thus, there exists $N\geq 1$ such that
\begin{equation}
|\varphi ^{i_j} (y,t_j)| \leq \frac{7}{8} \qquad \text{and} \qquad F_{\delta_{i_j}} ( \varphi ^{i_j} (y,t_j) )  \geq \frac{1}{10}
\end{equation}
for any $ y \in \bar B_{ \varepsilon _{i_j}/8 }(x_j)$ and $j>N$. Hence there exists $\eta =\eta (n) >0$ such that
\begin{equation}
2\eta \leq \int _{\bar B_{ \varepsilon _{i_j}/8 }(x_j)} \frac{ F_{\delta_{i_j}} ( \varphi ^{i_j} (y,t_j) ) }{\varepsilon _{i_j}} \rho _{x_j ,t_j +\varepsilon _{i_j}^2} (y,t_j) \, dy\leq \int _{\mathbb{R}^n} \rho _{x_j ,t_j +\varepsilon _{i_j}^2} (y,t_j) \, d\mu _{t_j} ^{i_j} (y), 
\end{equation}
where $\displaystyle \inf_{ y \in  \bar B_{ \varepsilon _{i_j}/8 }(x_j)  } \rho _{x_j , t_j +\varepsilon _{i_j} ^2} (y,t_j) \geq \frac{1 }{(4\pi)^{\frac{n-1}{2}} \varepsilon^{n-1} _{i_j} e^{\frac{1}{256}}  }>0$ is used. By the monotonicity formula \eqref{monoton2} we have
\begin{equation}
\int _{\mathbb{R}^n} \rho _{x_j ,t_j +\varepsilon _{i_j}^2} (y,t_j) \, d\mu _{t_j} ^{i_j} (y) \leq \int _{\mathbb{R}^n} \rho _{x_j ,t_j +\varepsilon _{i_j}^2} (y,s) \, d\mu _{s} ^{i_j} (y)
\end{equation}
for sufficiently large $j$. Hence we obtain
\begin{equation}
2\eta \leq \int _{\mathbb{R}^n} \rho _{x' ,t'} (y,s) \, d\mu _{s} (y).
\end{equation}
By \eqref{bound} and Lemma \ref{estilmanen}, for any $a>0$ there exists $\gamma_1 =\gamma_1 (a) >0$ such that for any $x\in \bar B_{\gamma_1 r} (x') $ we have
\begin{equation}
\begin{split}
2\eta &\leq \int_{\mathbb{R}^n} \rho_{x' ,t'}(y,s) \,d\mu _s (y) \leq (1+a) \int _{\mathbb{R}^n} \rho_{x,t'}(y,s) \,d\mu_s (y) +ac_3 D_1 \\
&= (1+a) \int _{\mathbb{R}^n} \rho_{x,s}(y,t) \,d\mu_s (y) +ac_3 D_1 \leq (1+a) \eta +ac_3 D_1 .
\end{split}
\end{equation} 
We remark that $\rho_{x,t'}(y,s)=\rho_{x,s}(y,t)$ by $t'-s=s-t=\frac{r^2}{2}$. Set $a:=\min \{ \frac{1}{4},\frac{\eta}{4c_3 D_1} \}$ and $\gamma _3: =\gamma_1 (a)$. Note that $\gamma _3$ depends only on $n$ and $D_1$. Then we have $\eta <0$. This is a contradiction to \eqref{assumption1}. Hence $(x',t')\not\in \spt \mu$.
\end{proof}

\begin{lemma}
Let $U\subset \mathbb{R}^n$ be open. There exists $c_6=c_6(n,D_1)>0$ such that 
\begin{equation}
\mathcal{H}^{n-1} (\spt \mu _t \cap U) \leq c_6 \liminf_{r\to 0} \mu _{t-r^2} (U) \quad \text{for} \quad t>0. 
\label{estsptmu}
\end{equation}
\end{lemma}
\begin{proof}
We only need to prove \eqref{estsptmu} for every compact set $K\subset U$. Let $X_t :=\spt \mu _t \cap K$. By an argument similar to that in Lemma \ref{lemeta}, for any $(x,t) \in X_t$ we have
\begin{equation}
2\eta \leq \int _{\mathbb{R}^n} \rho _{x ,t} (y,t-r^2) \, d\mu _{t-r^2} (y)
\end{equation}
for sufficiently small $r>0$. By \eqref{rhoest2}, for any $L>0$ we obtain
\begin{equation*}
\int _{\mathbb{R}^n \setminus B_{rL} (x)} \rho _{x,t} (y,t-r^2) \, d \mu _{t-r^2} (y) \leq 2^{n-1} e^{-\frac{3L^2}{8}}c_3 D_1.
\end{equation*}
Hence there exists $L=L(n,D_1)>0$ such that
\begin{equation*}
\eta \leq \int _{B_{rL}(x)} \rho _{x ,t} (y,t-r^2) \, d\mu _{t-r^2} (y).
\end{equation*}
Thus, by $\rho _{x,t} (\cdot,t-r^2) \leq \frac{1}{(4\pi)^{\frac{n-1}{2}} r^{n-1}} $ we obtain
\begin{equation}
(4\pi)^{\frac{n-1}{2}} r^{n-1} \eta \leq \mu _{t-r^2} (B_{rL} (x)). 
\label{4pi}
\end{equation}
Set $\mathcal{B} = \{ \bar{B} _{rL} (x) \subset U \, | \, x\in X_t \}$. Note that $\mathcal{B}$ is a covering of $X_t$ by closed balls centered at $x\in X_t$. By the Besicovitch covering theorem, there exists a finite sub-collection $\mathcal{B}_1, \mathcal{B}_2 ,\dots ,\mathcal{B}_{B(n)}$ such that each $\mathcal{B}_i$ is a disjoint set of closed balls and
\begin{equation}
X_t \subset \cup _{i=1}^{B(n)} \cup _{\bar{B}_{rL} (x_j) \in \mathcal{B}_i}\bar{B}_{rL} (x_j).
\label{xt}
\end{equation}
By \eqref{4pi} and \eqref{xt} we obtain
\begin{equation*}
\begin{split}
\mathcal{H}^{n-1} _{rL} (X_t) &\leq \sum _{i=1} ^{B(n)} \sum _{\bar{B}_{rL} (x_j) \in \mathcal{B}_i } \omega _{n-1} (rL)^{n-1}
\leq \frac{\omega _{n-1} L^{n-1} }{(4\pi)^{\frac{n-1}{2}} \eta } \sum _{i=1} ^{B(n)} \sum _{ \bar{B}_{rL} (x_j) \in \mathcal{B}_i} \mu _{t-r^2} (\bar{B}_{rL} (x_j))\\
&\leq \frac{\omega _{n-1} L^{n-1} }{(4\pi)^{\frac{n-1}{2}} \eta }  \sum _{i=1} ^{B(n)} \mu _{t-r^2} (U)
\leq \frac{\omega _{n-1} L^{n-1} }{(4\pi)^{\frac{n-1}{2}} \eta } B(n) \mu _{t-r^2} (U),
\end{split}
\end{equation*}
where $\mathcal{H}_{rL} ^{n-2+a }$ is the approximate Hausdorff measure of  $\mathcal{H}^{n-2+a }$. Set $c_6 :=\frac{\omega _{n-1} L^{n-1} }{(4\pi)^{\frac{n-1}{2}} \eta } B(n)$ which depends only on $n$ and $D_1$. Hence we obtain \eqref{estsptmu}.
\end{proof}
\begin{lemma}\label{lowerbound}
Let $\eta $ be as in Lemma \ref{lemeta}. Define
\[Z :=\Big\{ (x,t)\in \spt \mu \, \Big| \, t\geq 0, \ \limsup _{s \downarrow t} \int \rho _{y,s}(x,t) \, d\mu _s (y)< \eta/2 \Big\} \]
and
\[ Z _t :=Z \cap (\mathbb{R}^n \times \{ t \})\quad \text{for} \quad t\geq 0. \]
Then for $a >0$, $\mathcal{H}^{n-2+a}(Z_t ) =0$ for a.e. $t\geq 0$. Moreover we have $\mu (Z)=0$.
\end{lemma}
\begin{proof}Let $a>0$ and
\[ Z^{\tau} :=\left\{ (x,t)\in \spt \mu \, \Big| \, t\geq 0 , \ \int \rho _{y,s} (x,t) \, d\mu _s (y) < \eta  \ \text{for all} \ s\in (t,t+\tau ] \right\}. \]
First we prove $\mathcal{H}^{n-2+a}(Z_t ) =0$ for a.e. $t\geq 0$.
Note that $Z \subset \cup _{m=1}^\infty Z^{\tau _m} $ for some $\{ \tau _m \}_{m=1} ^\infty$ with $\tau _m \in (0,1 )$ and $\lim _{m \to \infty} \tau _m =0$. So we only need to prove $\mathcal{H}^{n-2+a}(Z^{\tau} _t)=0$ for any $\tau \in (0,1 )$, where $Z ^\tau _t: =Z^\tau \cap (\mathbb{R}^n \times \{ t \})$.
Set $r:=\sqrt{2(s-t)}$ and $t':=s+\frac{r^2}{2}$. For $(x,t) \in Z^\tau$, let $(x',t') \in \mathbb{R}^n \times [0,\infty )$ satisfy $|t'-t|\leq 2 \tau $ and $ |x'-x|\leq \gamma _3 r$. Then $(x',t') \not \in \spt \mu \subset Z^\tau$ by Lemma \ref{lemeta} and $ s-t\leq \tau$. Moreover, if $(x',t') \in Z^\tau$ then $ \int \rho _{y,s} (x,t) \, d\mu _s (y) \geq \eta  $ for any $x \in \bar B_{\gamma _3 r} (x')$ by $(x',t') \in \spt \mu$ and Lemma \ref{lemeta}. Therefore the relation 
\[ |t'-t|\leq 2 \tau \ \ \text{and} \ \ |x'-x|\leq \gamma _3 r \]
implies either $(x,t)\not \in Z^{\tau}$ or $(x',t')\not \in Z^{\tau}$. Hence for $(x,t) \in Z^\tau$ we have
\begin{equation}
P_{2\tau} (x,t)\cap Z^{\tau}=\{(x,t)\}.
\label{P}
\end{equation}
Here, $P_{2\tau} (x,t)$ is defined by
\[ P_{2\tau} (x,t):= \Big\{ (x',t') \, \Big| \, 2\tau \geq |t'-t| \geq \frac{|x'-x|^2}{\gamma _3 ^2} \Big\}. \]
Set
\[ Z^{\tau ,x_0,t_0}:=Z^{\tau}\cap (B_1 (x_0) \times [t_0-\tau , t_0+\tau]), \qquad x_0 \in \mathbb{R}^n , \ t_0 \geq 0. \]
Then there exists a countable set $K\subset \mathbb{R}^n \times [0,\infty)$ such that $Z^{\tau} \subset \cup_{(x_0,t_0) \in K} Z^{\tau ,x_0,t_0}$. Hence we only need to prove $\mathcal{H}^{n-2+a} (Z^{\tau ,x_0,t_0}_t )=0$ for a.e. $t \in (0,\infty)$, where $Z^{\tau ,x_0,t_0} _t =Z^{\tau ,x_0,t_0} \cap (\mathbb{R}^n \times \{t\})$. Remark that for any $x\in \mathbb{R}^n$ the set $\{x\}\times [t_0-\tau , t_0 + \tau ] \cap Z^{\tau ,x_0,t_0}$ has no more than one elements by \eqref{P}. Define $P: \mathbb{R}^{n+1} \to \mathbb{R}^{n}$ by $P(x,t)=(x,0) $, where $(x,t) \in \mathbb{R}^n \times [0,\infty)$. Let $\delta ' >0$ and cover the projection $P(Z^{\tau ,x_0,t_0})\subset B_1 (x_0) \times \{ 0 \}$ by $\{ B_{r_i} (x_i) \}_{i= 1} ^\infty$, where $(x_i ,0) \in P(Z^{\tau ,x_0,t_0})$, $r_i \leq \delta'$ and
\begin{equation*}
\sum _{i=1} ^\infty \omega _n r_i ^n \leq 2\mathcal{L}^n (B_1 (x_0)).
\end{equation*}
Let $(x_i ,t_i)$ be the point in $Z^{\tau ,x_0,t_0} $ corresponding to $x_i$. By \eqref{P} we have $Z^{\tau ,x_0,t_0}  \subset \sum _{t_i \in [t_0 - \tau, t_0 + \tau]} B_{r_i} (x_i) \times [t_i -\frac{r_i ^2}{\gamma_3 ^2}  , t_i + \frac{r_i ^2}{\gamma_3 ^2}]$. We compute that
\begin{equation*}
\begin{split}
&\int_{t_0 -\tau } ^{t_0 + \tau} \mathcal{H}_{\delta ' } ^{n-2+a } (Z^{\tau ,x_0,t_0} _t ) \, dt 
\leq \int_{t_0 -\tau } ^{t_0 + \tau}\sum _{i=1}^\infty \sum _{ t \in [t_i -r_i^ 2 /\gamma_3 ^2 , t_i +r_i^ 2 /\gamma_3 ^2 ] } \omega _{n-2+a } r_i ^{n-2+a } \, dt \\
= &\sum _{i= 1}^\infty \int _{t_i -r_i^ 2 /\gamma_3 ^2 } ^{t_i +r_i^ 2 /\gamma_3 ^2 } \omega _{n-2+a } r_i ^{n-2+a } \, dt = \sum _{i=1}^\infty \frac{2\omega _{n-2+a}}{\gamma_3 ^2} r_i ^{n+a}
\leq \frac{4\omega_{n-2+a}}{\gamma _3 ^2 \omega_n} (\delta ')^a  \mathcal{L} ^n (B_1 (x_0)),
\end{split}
\end{equation*}
where $\mathcal{H}_{\delta ' } ^{n-2+a }$ is the approximate Hausdorff measure of  $\mathcal{H}^{n-2+a }$. Then $\delta '\to 0$ implies that
\[ \int_{t_0 -\tau } ^{t_0 + \tau} \mathcal{H} ^{n-2+a } (Z^{\tau ,x_0,t_0} _t) \, dt =0. \]
Hence we obtain $\mathcal{H}^{n-2+a}(Z_t ) =0$ for a.e. $t\in [0,\infty)$. 
On the other hand, we compute that
\begin{equation*}
\begin{split}
&\int_{t_0 -\tau } ^{t_0 + \tau} \mu _t (Z^{\tau ,x_0,t_0} _t ) \, dt 
\leq \int_{t_0 -\tau } ^{t_0 + \tau}\sum _{i=1}^\infty \sum _{ t \in [t_i -r_i^ 2 /\gamma_3 ^2 , t_i +r_i^ 2 /\gamma_3 ^2 ] } c_3 D_1 r_i ^{n-1} \, dt \\
= &\sum _{i= 1}^\infty \int _{t_i -r_i^ 2 /\gamma_3 ^2 } ^{t_i +r_i^ 2 /\gamma_3 ^2 }  c_3 D_1r_i ^{n-1}  \, dt = \sum _{i=1}^\infty \frac{2c_3 D_1}{\gamma_3 ^2} r_i ^{n+1}
\leq \frac{4c_3 D_1}{\gamma _3 ^2 \omega_n} \delta \mathcal{L} ^n (B_1 (x_0)).
\end{split}
\end{equation*}
Then $\delta ' \to 0$ implies that $\int_{t_0 -\tau } ^{t_0 + \tau} \mu _t (Z^{\tau ,x_0,t_0} _t ) \, dt=0$.
Thus we obtain $\mu (Z)=0$.
\end{proof}

\begin{lemma}\label{lem-xi}
For any $(y,s)\in \mathbb{R}^n \times [0,\infty)$ we have
\begin{equation}
\int _0 ^s \int_{\mathbb{R}^n} \frac{1}{2(s-t)}\rho _{y,s}(x,t) \, d|\xi _t ^{i}| (x) dt \leq c_4 D_1\quad \text{for} \quad i\geq 1.
\label{absolute}
\end{equation}
\end{lemma}

\begin{proof}
We compute that
\begin{equation*}
\begin{split}
&\frac{d}{dt}\int_{\mathbb{R}^n}  \rho _{y,s}(x,t) \, d\mu_t ^i(x) \leq  
\frac{1}{2(s-t)}\int _{\mathbb{R}^n} \rho _{y,s}(x,t) \, d\xi_t ^i (x)= - \frac{1}{2(s-t)}\int_{\mathbb{R}^n}  \rho _{y,s}(x,t) \, d|\xi_t ^i|(x),
\end{split}
\label{absolute2}
\end{equation*}
where Proposition \ref{negative} and \eqref{monoton} are used. Therefore we have
\begin{equation*}
\begin{split}
&\int _0 ^s \frac{1}{2(s-t)}\int_{\mathbb{R}^n}  \rho _{y,s}(x,t) \, d|\xi_t ^i| (x)dt \leq -\int _0 ^s \frac{d}{dt} \int _{\mathbb{R}^n} \rho _{y,s}(x,t) \, d\mu_t ^i(x) dt\\
= & \int_{\mathbb{R}^n}  \rho _{y,s} (x,0) \, d\mu _0 ^i (x) -\lim _{t \uparrow s}  \int _{\mathbb{R}^n} \rho _{y,s} (x,t) \, d\mu _t ^i (x)\leq  \int_{\mathbb{R}^n}  \rho _{y,s} (x,0) \, d\mu _0 ^i(x)\leq c_4D_1,
\end{split}
\label{absolute3}
\end{equation*}
where \eqref{density1} is used. Hence we obtain \eqref{absolute}.
\end{proof}
We may assume that there exists a Radon measure $\xi_t$ such that $\xi ^i _t \to \xi _t$ as Radon measures. Define $d\xi := d \xi _t  dt$. Next we prove the vanishing of the discrepancy measure $\xi$.
\begin{lemma}\label{vanish}
Assume that $\varphi^{\varepsilon_i,\delta_j}$ and $\mu ^{\varepsilon_i,\delta_j} _t$ satisfy all the assumptions of Section 2 and $\mu ^i _t =\mu ^{\varepsilon _i,\delta_i} _t \to \mu _t$ and $\xi ^i _t =\xi ^{\varepsilon _i,\delta_i} _t \to \xi _t$ for any $t\geq 0$ as Radon measures. Then  $\xi=0$.
\end{lemma}
\begin{proof}
By \eqref{absolute} and $\xi^{i} _t \to \xi _t$ we have
\begin{equation*}
\int_{\mathbb{R}^n \times (0,s)} \frac{1}{2(s-t)} \rho _{y,s}(x,t) \, d |\xi|(x,t) \leq c_4 D_1.
\end{equation*}
Let $R$ and $T$ be positive numbers. We integrate with  the measure $d\mu _s ds$
\begin{equation*}
\begin{split}
&\int _{B_R (0) \times [0,T+1]} \int_{\mathbb{R}^n \times (0,s)}\frac{1}{2(s-t)} \rho _{y,s}(x,t) \, d|\xi| (x,t) d\mu _s (y)ds \\
\leq & \int _{B_R (0)\times [0,T+1]} c_4 D_1 d\mu _s (y)ds \leq c_3 c_4 D_1^2(T+1)R^{n-1}<\infty,
\end{split}
\end{equation*}
where \eqref{density} is used. By Fubini's theorem we obtain
\begin{equation*}
\begin{split}
\int _{\mathbb{R}^n \times [0,T+1] } \Big( \int _t ^{T+1} \frac{1}{2(s-t)} \int _{B_R (0)} \rho _{(y,s)} (x,t) \, d\mu _s (y) ds\Big) d|\xi|(x,t) \\
\leq c_3 c_4 D_1^2(T+1)R^{n-1}.
\end{split}
\end{equation*}
Hence there exists $c_7 =c_7 (x,t)<\infty$ such that
\begin{equation}
\int _t ^{t+1} \frac{1}{2(s-t)} \int _{B_R (0)} \rho _{(y,s)} (x,t) \, d\mu _s (y) ds  \leq c_7 (x,t) <\infty
\label{vanish1}
\end{equation}
for $|\xi|$-a.e. $(x,t)\in \mathbb{R}^n \times [0,T]$.

Let $x\in B_{\frac{R}{2}} (0)$ and $T>s>t>0$. We compute
\begin{equation}
\begin{split}
\int _{\mathbb{R}^n} \rho _{y,s} (x,t) \, d\mu _s (y) \leq \int _{B_R (0) } \rho _{y,s}(x,t) \, d\mu _s (y) + 2^{n-1} e^{-\frac{3}{8} \frac{(R/2)^2}{2(s-t)} } D_1,
\end{split}
\label{vanish2}
\end{equation}
where \eqref{rhoest2} is used.
By \eqref{vanish1} and \eqref{vanish2} we obtain
\begin{equation}
\begin{split}
&\int _t ^{t+1} \frac{1}{2(s-t)} \int _{\mathbb{R}^n} \rho _{y,s} (x,t) \, d\mu _s (y) ds \\
\leq & c_7 (x,t) + \int _t ^{t+1} \frac{1}{2(s-t)} 2^{n-1} e^{-\frac{3}{64} \frac{R^2}{s-t}} D_1 \, ds <\infty
\end{split}
\label{vanish3}
\end{equation}
for $|\xi|$-a.e. $(x,t)\in B_{\frac{R}{2}} \times [0,T]$ and for any $R>0$ and $T>0$. Hence \eqref{vanish3} holds for $|\xi|$-a.e. $(x,t)\in \mathbb{R}^n \times [0,\infty)$. Set
\[ h(s)=h_{x,t}(s):= \int _{\mathbb{R}^n} \rho _{y,s} (x,t) \, d\mu _s (y), \quad (x,t) \in \mathbb{R}^n \times [0,\infty ). \]
Next, we claim that
\begin{equation}
\lim _{s\to t} h_{x,t}(s) =0 \qquad |\xi|\text{-a.e.} \ (x,t) \in \mathbb{R}^n \times [0,\infty).
\label{vanish4}
\end{equation}
Define
\[ A=\left\{ (x,t)\in \mathbb{R}^n \times [0,\infty) \, \Big| \, \int _t ^{t+1} \frac{1}{2(s-t)} \int _{\mathbb{R}^n} \rho _{y,s} (x,t) \, d\mu _s (y) ds <\infty\right\} . \]
Note that $|\xi|(A^c)=0$. Fix $(x,t) \in A$ and set $\lambda : = \log (s-t)$. Then we have
\begin{equation}
\int _{-\infty} ^0 h(t+e^\lambda)\, d\lambda=\int _t ^{t+1} \frac{1}{s-t} \int _{\mathbb{R}^n} \rho _{y,s} (x,t) \, d\mu _s (y) ds <\infty.
\label{vanish5}
\end{equation}
Set $\kappa \in (0,1]$. By \eqref{vanish5} there exists a sequence $\{ \lambda _i \}_{i=1}^\infty$ such that
\begin{equation}
\lambda _i \downarrow  -\infty, \qquad \lambda_i -\lambda_{i+1}\leq \kappa, \qquad h(t+e^{\lambda _i}) \leq \kappa.
\label{vanish6}
\end{equation}
Fix $\lambda \in (-\infty , \lambda_1 ]$ and choose $i$ such that $\lambda \in [\lambda_i , \lambda _{i-1} )$. Then by \eqref{monoton2} we have
\begin{equation}
\begin{split}
h(t+e^\lambda )&= \int \rho _{y,t+e^\lambda } (x,t) \, d\mu _{t+e^\lambda} (y)
= \int \rho _{y,t+2 e^\lambda } (x,t+e^\lambda ) \, d\mu _{t+e ^\lambda} (y)\\
& \leq \int \rho _{y,t+2 e^\lambda } (x,t+e^{\lambda _i} ) \, d\mu _{t+e ^{\lambda_i} } (y)
= \int \rho ^R _x \, d\mu _{t+e^{\lambda _i}},
\end{split}
\label{vanish7}
\end{equation}
where $\frac{R^2}{2}= 2e^\lambda -e^{\lambda _i}$. On the other hand, by \eqref{vanish6} we have
\begin{equation}
\kappa \geq h(t+e^{\lambda _i}) = \int \rho _{y,t+e^{\lambda_i}} (x,t) \, d\mu _{t+e^{\lambda_i}} (y) = \int \rho _x ^r \, d\mu _{t+e^{\lambda_i}},
\label{vanish8}
\end{equation}
where $\frac{r^2}{2}=e^{\lambda _i}$. Remark that there exists $c_8 >0$ such that
\[ 1\leq \frac{R}{r} =\sqrt{2e^{\lambda-\lambda _i} -1} \leq 1+ c_8 \kappa. \]

Let $1>a >0$ and $\kappa = \min \{ a , \gamma_2(a) /c_8 \} $ where $\gamma_2= \gamma _2 (a)$ is defined by Lemma \ref{estilmanen}. Then for $\lambda \leq \lambda_1 (a )$,
\begin{equation*}
\begin{split}
&h(t+e^\lambda) \leq \int \rho _x ^R \, d\mu _{t+e^{\lambda _i}}\\
\leq & (1+a) \int \rho ^r _x \, d\mu_{t+e^{\lambda _i}} +a D_1
\leq 2\kappa +a D_1,
\end{split}
\end{equation*}
where \eqref{vanish7} and \eqref{vanish8} are used. Passing $a \to 0$ we obtain
\begin{equation}
\lim _{s\downarrow t} h_{x,t}(s) =0 \qquad \text{for} \ (x,t)\in A.
\label{vanish9}
\end{equation}
Thus we have \eqref{vanish4}. On the other hand, by Lemma $\ref{lowerbound}$ we obtain
\begin{equation}
\limsup _{s\downarrow t} h_{x,t}(s) \geq \frac{\eta}{2} >0 \qquad \text{for} \ \mu\text{-a.e.} \ (x,t).
\label{vanish10}
\end{equation}
Hence by \eqref{vanish4} and \eqref{vanish10} we obtain
\[ 0\geq \limsup _{s\downarrow t} h_{x,t}(s) \geq \frac{\eta}{2} \qquad \text{for} \ |\xi|\text{-a.e.}, \]
where $|\xi| \ll \mu $ are used.
Thus we have $|\xi|=0$.
\end{proof}
\section{Proof of main results}
Let $\{ \varepsilon _i \}_{i=1} ^\infty$ and $\{ \delta _i \}_{i=1} ^\infty$ be positive sequences with $\varepsilon _i ,\delta_i \downarrow 0$. Set $\tilde\varphi ^i \in C^{2,\alpha}_{loc} (\mathbb{R}^n)$ for $i\in\mathbb{N}$. Define measures $\tilde\mu ^i$, $\tilde\xi^i$ and $\tilde V^i$ by
\begin{equation*}
\tilde \mu ^i (\phi) :=\int _{\mathbb{R}^n} \phi \Big(\frac{\varepsilon _i |\nabla \tilde\varphi ^i|^2}{2} +\frac{F_{\delta _i}(\tilde\varphi ^i)}{\varepsilon _i}\Big) \, dx \ \ \text{and} \ \ \tilde\xi ^i (\phi) :=\int _{\mathbb{R}^n} \phi \Big(\frac{\varepsilon _i |\nabla \tilde\varphi ^i|^2}{2} -\frac{F_{\delta _i}(\tilde\varphi ^i)}{\varepsilon _i}\Big) \, dx
\end{equation*}
for $\phi \in C_c (\mathbb{R}^n)$, and
\begin{equation*}
\tilde V^i (\psi) :=\int _{\{ x \, | \,  |\nabla \tilde\varphi ^i (x) |\not =0 \}} \psi (x,I-\nu ^i \otimes \nu ^i ) \Big(\frac{\varepsilon _i |\nabla \tilde\varphi ^i|^2}{2} +\frac{F_{\delta _i}(\tilde\varphi ^i)}{\varepsilon _i}\Big) \, dx
\end{equation*}
for $\psi \in C_c(\mathbb{R}^n \times G_{n-1}(\mathbb{R}^n))$, where $\nu ^i :=\frac{\nabla \tilde\varphi ^i}{|\nabla \tilde\varphi ^i|}$. Note that $\tilde V^i \in \mathbf{V}_{n-1} (\mathbb{R}^n)$ and $\|\tilde V^i \|=\tilde \mu ^i$. For $\phi \in C_c ^2 (\mathbb{R}^n)$, define
\begin{equation*}
\mathcal{B} ^i (\tilde\varphi ^i, \phi): =\int _{\mathbb{R}^n} - \varepsilon _i \phi \Big(-\Delta \tilde\varphi ^i +\frac{F' _{\delta _i}(\tilde\varphi ^i)}{\varepsilon _i ^2}\Big)^2 +\varepsilon _i \nabla \phi \cdot \nabla \tilde\varphi ^i \Big(-\Delta \tilde\varphi ^i +\frac{F' _{\delta _i}(\tilde\varphi ^i)}{\varepsilon _i ^2}\Big) \, dx.
\end{equation*}
The following lemma is obtained in the same manner as Lemma 9.3 of \cite{Ilmanen}. So we omit the proof.
\begin{lemma}\label{lowersemiconti}
For $\phi \in C_c ^2 (\mathbb{R}^n)$ we assume that 
\begin{enumerate}
\item $\tilde \mu ^i \to \tilde \mu$ as Radon measures on $\mathbb{R}^n$, 
\item $\tilde \xi ^i$ is non-positive measure for $i\in \mathbb{N}$,
\item $|\tilde \xi ^i|\lfloor {\{ \phi >0 \}} \to 0$ as Radon measures on $\mathbb{R}^n$,
\item there exists $C>0$ such that $\mathcal{B} ^i (\tilde \varphi ^i, \phi)\geq -C$ for $i\in \mathbb{N}$,
\item $\mathcal{H}^{n-1} (\spt \tilde \mu \cap \{ \phi >0 \})<\infty$.
\end{enumerate}
Then the following hold:
\begin{enumerate}
\item $\tilde \mu \lfloor {\{ \phi >0 \}}$ is $(n-1)$-rectifiable.
\item There exists $\tilde V \in \mathbf{V} _{n-1} (\mathbb{R}^n)$ such that $\tilde V^i \lfloor {\{ \phi >0 \}} \to\tilde V$ and $\| \tilde V \| = \tilde \mu \lfloor {\{ \phi >0\}}$.
\item For any $Y \in C_c ^1 (\{ \phi >0 \};\mathbb{R}^n )$ we have
\begin{equation}
\delta \tilde V ( Y ) =\lim _{i\to \infty} \int -\varepsilon _i Y\cdot \nabla \tilde \varphi ^i \Big(-\Delta \tilde\varphi ^i +\frac{F' _{\delta _i}(\tilde\varphi ^i)}{\varepsilon _i ^2}\Big) \, dx.
\end{equation}
\item There exists the generalized mean curvature vector $H $ for $\tilde V$ with
\begin{equation}
\int _{\mathbb{R}^n} \psi |H|^2 \, d\tilde\mu \leq \frac{2}{\pi}\liminf _{i\to \infty} \int _{\mathbb{R}^n} \varepsilon _i \psi \Big(-\Delta \tilde\varphi ^i +\frac{F' _{\delta _i}(\tilde\varphi ^i)}{\varepsilon _i ^2}\Big)^2 \, dx<\infty
\end{equation}
for $\psi \in C_c ^2 (\{ \phi >0 \} ;\mathbb{R}^+)$.
\item 
\begin{equation}
\limsup _{i\to \infty} \mathcal{B}^i (\tilde \varphi ^i ,\phi) \leq\mathcal{B}(\tilde \mu,\phi).
\end{equation}
\end{enumerate}
\end{lemma}
\begin{flushleft}
\textbf{Proof of Theorem \ref{main}}
\end{flushleft}

First we prove Brakke's inequality. Let $\varphi ^{\varepsilon _i ,\delta_j}  \in C^{2,\alpha}_{loc}(\mathbb{R}^n\times (0,\infty))$ be as in Theorem \ref{main}. Then by Proposition \ref{existencemeasure} there exist subsequences $\{ \varepsilon _{i_k} \}_{k=1} ^\infty, \{ \delta _{j_k} \}_{k=1} ^\infty $ and $\{ \mu _t \} _{t\geq 0} $ such that  \eqref{result1}, \eqref{result2} and \eqref{result3} hold. By Lemma \ref{vanish}, there exist subsequences $\{ \varepsilon _{i_k} \}_{k=1} ^\infty$ and $ \{ \delta _{j_k} \}_{k=1} ^\infty $ (denoted by the same index) such that
\begin{equation}
\xi^{\varepsilon _{i_k},\delta _{j_k} } \to 0 \quad \text{as Radon measures on} \quad \mathbb{R}^n \times [0,\infty).
\label{limxi}
\end{equation}

Set $\varphi ^k=\varphi ^{\varepsilon _{i_k},\delta_{j_k}}, \ \mu _t ^k=\mu _t ^{\varepsilon_{i_k},\delta_{j_k}}, \ \xi _t ^k=\xi _t ^{\varepsilon_{i_k},\delta_{j_k}} $ and $\xi ^k=\xi ^{\varepsilon_{i_k},\delta_{j_k}}$. Let $t_0 \geq 0$ and $\phi \in C_c ^2 (\mathbb{R}^n ; \mathbb{R}^+)$. If $\overline{D}_{t} \mu _t (\phi) \Big| _{t=t_0} =-\infty$, then \eqref{brakkeineq} holds. Therefore we assume that 
\begin{equation}
C_0:=\overline{D}_{t} \mu _t (\phi) \Big| _{t=t_0}>-\infty .
\label{d0}
\end{equation}

Then there exist $\{ h_q \}_{q=1} ^\infty$ and $\{ t_q \}_{q=1} ^\infty$ such that $h_q \downarrow 0, t_q \to t_0$ as $q \to \infty$ and
\[ C_0 -h_q \leq \frac{\mu_{t_q}(\phi)  -\mu _{t_0} (\phi) }{t_q-t_0}\quad \text{for} \quad q\geq1. \]
We may assume that $t_q >t_0$ for any $q\geq 1$. (The other case is similar.)

By $\mu _t ^k \to \mu_t$ and \eqref{limxi} there exists a subsequence $\{ k_q \}_{q=1} ^\infty$ such that
\begin{equation}
C_0 -2h_q \leq \frac{\mu_{t_q} ^{k_q} (\phi)  -\mu _{t_0} ^{k_q} (\phi) }{t_q-t_0}=\frac{1}{t_q -t_0} \int _{t_0} ^{t_q} \frac{d}{dt} \mu _t ^{k_q} (\phi) \, dt
\label{dh}
\end{equation}
and
\begin{equation}
\int _{ \{ \phi >0 \} \times [t_0 ,t_q] } d|\xi ^{k_q}| \leq h_q ^2 (t_q -t_0)
\label{hq2}
\end{equation}
for $q\geq 1$.
By Proposition \ref{densityprop} and Lemma \ref{noninc} there exists $C_1 = C_1(n,\phi,D_1) >0$ such that
\[ \frac{d}{dt} \mu _t ^k (\phi) \leq C_1 \quad \text{for} \ k \geq 1, \ t\geq 0. \]
We may assume $C_1>C_0$. Set
\[ Z_q =\{ t\in[t_0,t_q] \, | \, \frac{d}{dt} \mu _t ^{k_q} (\phi) \geq C_0 -3h_q \} \quad \text{for} \ q \geq 1. \]
By \eqref{dh} we have
\begin{equation*}
C_0 -2h_q \leq \frac{1}{t_q -t_0} \int _{[t_0,t_q] \setminus Z_q} C_0 -3h_q \, dt +\frac{1}{t_q -t_0} \int _{Z_q} C_1 \, dt.
\end{equation*}
Hence we obtain
\begin{equation}
|Z_q| \geq \frac{(t_q-t_0) h_q}{C_1 -C_0 +3 h_q} \geq \frac{(t_q-t_0) h_q}{2(C_1 -C_0)}
\label{zq1}
\end{equation}
for sufficiently large $q\geq 1$. By \eqref{hq2} we have
\begin{equation}
|Z_q| \inf_{t\in Z_q} |\xi _t ^{k_q}|(\{ \phi >0 \}) \leq h_q ^2 (t_q -t_0).
\label{zq2}
\end{equation}
By \eqref{deri}, \eqref{zq1} and \eqref{zq2} for any $q\geq 1$ there exists $s_q \in Z_q$ such that
\begin{equation}
C_0 -3h_q \leq \frac{d}{dt} \mu_t ^{k_q} (\phi)\Big|_{t=s_q} =\mathcal{B}^{k_q}(\varphi^{k_q}(\cdot,s_q), \phi)
\label{dhq}
\end{equation}
and
\begin{equation}
 |\xi _{s_q} ^{k_q}|(\{ \phi >0 \}) \leq 3(C_1 -C_0)h_k.
\label{sq}
\end{equation}
Assume that the subsequence $\{ \mu _{s_q}  ^{k_q} \}_{q=1}^\infty$ converges to a Radon
measure $\tilde \mu$. By Lemma \ref{noninc} and \eqref{d0}, it is possible to prove (see \cite[7.1]{ilmanen1994}) that
\begin{equation}
\tilde\mu \lfloor {\{\phi >0\}} =\mu _{t_0} \lfloor {\{ \phi >0 \}}.
\label{tilde}
\end{equation}
Hence, by Lemma \ref{lowersemiconti}, \eqref{estsptmu}, \eqref{d0}, \eqref{dhq}, \eqref{sq} and \eqref{tilde} we have
\begin{equation*}
\overline{D}_{t} \mu _{t} (\phi) \Big|_{t=t_0} \leq \mathcal{B} (\mu _{t_0} , \phi).
\end{equation*}
Therefore by this and Proposition \ref{prop1} (2), $\{ \mu _t \}_{t\geq 0}$ is a global solution for Brakke's mean curvature flow with initial data $\mu _0 =\frac{\pi}{2} \mathcal{H}^{n-1} \lfloor \Gamma_0$. Thus we obtain (a).

Next we prove (b). Set $w^k := \Phi _k \circ \varphi ^k$, where $\Phi _k (s) := \sigma_{\delta_{j_k}} ^ {-1} \int _{-(1-\delta_{j_k})^{-1}} ^s \sqrt{2F_{\delta _{j_k}} (y) } \, dy $ and $\varphi ^k := \varphi^{\varepsilon_{i_k},\delta_{j_k}}$. Note that $\Phi _k (-(1-\delta_{j_k})^{-1}) =0$ and $\Phi _k ((1-\delta_{j_k})^{-1}) =1$. We denote $\varepsilon = \varepsilon _{i_k}$ and $\delta=\delta_{j_k}$. We compute that
\[ |\nabla w^k |=\sigma _{\delta} ^{-1} |\nabla \varphi ^k| \sqrt{2F_\delta (\varphi ^k) } \leq \sigma _{\delta} ^{-1} \Big( \frac{\varepsilon |\nabla \varphi ^k |^2 }{2} +\frac{F_\delta (\varphi ^k)}{\varepsilon} \Big). \]
Hence by \eqref{boundrn} we have
\begin{equation}
\int _{\mathbb{R}^n} |\nabla w^k (\cdot,t) | \, dx \leq \int _{\mathbb{R}^n} \sigma _{\delta} ^{-1} \Big( \frac{\varepsilon |\nabla \varphi ^k |^2 }{2} +\frac{F_\delta (\varphi ^k)}{\varepsilon} \Big) \, dx \leq \sigma _\delta ^{-1} D_1
\label{bv1}
\end{equation}
for $t\geq 0$. Fix $T>0$. By the similar argument and \eqref{boundrn} we obtain
\begin{equation}
\begin{split}
&\int _0 ^T \int_{\mathbb{R}^n} |\partial _t w^k| \, dxdt \leq \sigma_{\delta} ^{-1} \int _0 ^T \int_{\mathbb{R}^n} \Big( \frac{\varepsilon |\partial _t \varphi ^k |^2 }{2} +\frac{F_\delta (\varphi ^k)}{\varepsilon} \Big) \, dxdt \\
\leq & \frac{\sigma_{\delta} ^{-1}}{2}\int _0 ^T \int _{\mathbb{R}^n} \varepsilon \Big( \Delta \varphi ^k -\frac{F_\delta ' (\varphi ^k)}{\varepsilon ^2} \Big) ^2 \, dxdt  +\sigma_{\delta} ^{-1}D_1T \leq \sigma_{\delta} ^{-1} D_1(1+T).
\end{split}
\label{bv2}
\end{equation}
By \eqref{bv1} and \eqref{bv2}, $\{ w^k \}_{k=1} ^\infty$ is bounded in $BV  (\mathbb{R}^n \times [0,T])$. By the standard compactness theorem and the diagonal argument there is subsequence $\{ w^k \}_{k=1} ^\infty$ (denoted by the same index) and $w\in BV_{loc} (\mathbb{R}^n \times [0,\infty))$ such that
\begin{equation}
w^k \to w \quad \text{in} \ L^1 _{loc} (\mathbb{R}^n \times [0,\infty))
\label{wconv}
\end{equation}
and a.e. pointwise. We denote $\varphi (x,t) := \lim _{k\to \infty} (1+\Phi _k ^{-1} \circ w^k (x,t))/2$. Then we have
\[ \varphi ^k \to 2\varphi -1 \quad \text{in} \ L^1 _{loc} (\mathbb{R}^n \times [0,\infty)) \]
and a.e. pointwise. Hence we obtain (b1). By Proposition \ref{prop1} (3) we obtain (b2). We have $\varphi ^k \to \pm 1$ a.e. and $\varphi =1$ or $=0$ a.e. on $\mathbb{R}^n \times [0,\infty)$ by the boundedness of $\int _{\mathbb{R}^n} \frac{F_{\delta} (\varphi ^k) }{\varepsilon} \, dx$. Moreover $\varphi =w$ a.e. on $\mathbb{R}^n \times [0,\infty)$. Thus $\varphi \in BV_{loc} (\mathbb{R}^n \times [0,\infty))$. For any bounded open set $U \subset \mathbb{R}^n$ and a.e. $0\leq t_1 < t_2<T$ we have
\begin{equation}
\begin{split}
&\int_U |\varphi (\cdot ,t_2) -\varphi (\cdot ,t_1)| \, dx =\lim _{k\to \infty} \int _{U} |w^k (\cdot ,t_2) -w^k (\cdot ,t_1)| \, dx \\
\leq & \liminf _{k \to \infty} \int_U \int _{t_1} ^{t_2} |\partial _t w^k| \, dtdx \leq \liminf_{k\to \infty} \int_{\mathbb{R}^n} \int _{t_1} ^{t_2} \Big( \frac{\varepsilon |\partial _t \varphi ^k |^2 }{2}\sqrt{t_2 -t_1} +\frac{F_\delta (\varphi ^k)}{\varepsilon \sqrt{t_2 -t_1}} \Big) \, dtdx \\
\leq & C_2 D_1 \sqrt{t_2 -t_1},
\end{split}
\label{holder}
\end{equation}
where $C_2= C_2 (n,T)>0$. By \eqref{holder} and $|\Omega _0 ^+|<\infty$, $\varphi (\cdot ,t) \in L^1 (\mathbb{R}^n)$ for a.e. $t\geq 0$. By this and \eqref{holder}, we may define $\varphi (\cdot ,t)$ for any $t\geq 0$ such that $\varphi \in C^{\frac{1}{2}} _{loc} ([0,\infty) ; L^1 (\mathbb{R}^n))$. Hence we obtain (b3). For $\phi \in C_c (\mathbb{R}^n ;\mathbb{R}^+)$ and $t\geq 0$ we compute that
\begin{equation*}
\begin{split}
&\int _{\mathbb{R}^n} \phi \, d\| \nabla \varphi (\cdot , t) \| \leq \liminf _{k \to \infty} \int _{\mathbb{R}^n} \phi |\nabla w^k | \, dx\\
\leq & \lim _{k\to \infty}\sigma ^{-1} _{\delta_{j_k}} \int _{\mathbb{R}^n} \phi \Big( \frac{\varepsilon_{i_k} |\nabla \varphi ^k |^2 }{2} +\frac{F_{\delta_{j_k}} (\varphi ^k)}{\varepsilon_{i_k}} \Big) \, dx = \frac{2}{\pi} \int _{\mathbb{R}^n} \phi \, d\mu _t.
\end{split}
\end{equation*}
Hence we obtain (b4).
\qed

\end{document}